\documentclass[reqno, 10pt]{amsart}
\usepackage{amssymb}
\usepackage{amsmath}
\usepackage[mathscr]{euscript}
\usepackage{hyperref}
\usepackage{graphicx}
\usepackage{mathabx}
\usepackage{comment}

\usepackage{amssymb}
\usepackage{hyperref}
\RequirePackage[dvipsnames]{xcolor} 
\definecolor{halfgray}{gray}{0.55} 
\definecolor{webgreen}{rgb}{0,0.5,0}
\definecolor{webbrown}{rgb}{.6,0,0} \hypersetup{%
  colorlinks=true, linktocpage=true, pdfstartpage=3,
  pdfstartview=FitV,%
  breaklinks=true, pdfpagemode=UseNone, pageanchor=true,
  pdfpagemode=UseOutlines,%
  plainpages=false, bookmarksnumbered, bookmarksopen=true,
  bookmarksopenlevel=1,%
  hypertexnames=true,
  pdfhighlight=/O,
  urlcolor=webbrown, linkcolor=RoyalBlue,
  citecolor=webgreen, 
  pdftitle={},%
  pdfauthor={},%
  pdfsubject={2000 MAthematical Subject Classification: Primary:},%
  pdfkeywords={},%
  pdfcreator={pdfLaTeX},%
  pdfproducer={LaTeX with hyperref}%
}

\theoremstyle{plain}
\numberwithin{equation}{section}
\newtheorem{theorem}{Theorem}
\newtheorem{Claim}{Claim}

\newtheorem{lemma}{Lemma}

\theoremstyle{definition}
\newtheorem{definition}{Definition}
\newtheorem{remark}{Remark}

\def\Z{\mathbb{Z}}

\begin{document}
\title[Shadowing and hyperbolicity
for delay difference equations]{Shadowing and hyperbolicity for linear delay difference equations}

\author[L. Backes]{Lucas Backes}
\address{\noindent Departamento de Matem\'atica, Universidade Federal do Rio Grande do Sul, Av. Bento Gon\c{c}alves 9500, CEP 91509-900, Porto Alegre, RS, Brazil.}
\email{lucas.backes@ufrgs.br}

\author[D. Dragi\v cevi\'c]{Davor Dragi\v cevi\'c}
\address{Faculty of Mathematics, University of Rijeka, Radmile Matej\v ci\' c 2, 51000 Rijeka, Croatia}
\email{ddragicevic@math.uniri.hr}

\author[M. Pituk]{Mih\'aly Pituk}
\address{Department of Mathematics, University of Pannonia, Egyetem \'ut 10, 8200 Veszpr{\'e}m, Hungary;
HUN--REN--ELTE Numerical Analysis and Large Networks Research Group, Budapest, Hungary}
\email{pituk.mihaly@mik.uni-pannon.hu}



\begin{abstract}
\vskip20pt
It is known that hyperbolic linear delay difference equations are shadowable on the half-line. In this paper, we prove the converse and hence the equivalence between hyperbolicity and the positive shadowing property for the 
following two classes of linear delay difference equations: 
(a)~for nonautonomous 
equations
with finite delays and uniformly bounded compact coefficient operators in 
Banach spaces, (b)~for 
Volterra difference equations with infinite delay in finite dimensional spaces.
\end{abstract}

\maketitle

\section{Introduction}\label{intro}

Let $\mathbb Z$ and $\mathbb C$ denote the set of integers and the set of complex numbers, respectively. For $k\in\mathbb Z$, define $\mathbb Z^+_k=\{\,n\in\mathbb Z:n\geq k\,\}$ and $\mathbb Z^-_k=\{\,n\in\mathbb Z:n\leq k\,\}$.
Throughout the paper, we shall assume as a standing assumption that $(X,|\cdot|)$ is a Banach space. The symbol $\mathcal L(X)$ will denote the space of all bounded linear operators $A\colon X\to X$ equipped with the operator norm, 
$
|A|=\sup_{|x|=1}|Ax|
$
for $A\in\mathcal L(X)$.  

Consider the linear autonomous difference equation
\begin{equation}\label{ade}
	x(n+1)=Ax(n), 
\end{equation}
where $A\in\mathcal L(X)$. Given $\delta\geq0$, by a \emph{$\delta$-pseudosolution} of Eq.~\eqref{ade} on~$\mathbb Z^+_0$, we mean a function $y\colon \mathbb Z^+_0\to X$ such that 
$$
\sup_{n\geq 0}
|y(n+1)-Ay(n)|\leq\delta.
$$
Note that for $\delta=0$ the pseudosolution becomes a true solution of Eq.~\eqref{ade}
on $\mathbb Z^+_0$.
We say that Eq.~\eqref{ade} is \emph{shadowable} on~$\mathbb Z^+_0$ or that it has the \emph{positive shadowing property} if for every $\epsilon>0$ there exists
$\delta>0$ such that for every $\delta$-pseudosolution~$y$ of~\eqref{ade} on~$\mathbb Z^+_0$ there exists a true solution $x$ of~\eqref{ade} on~$\mathbb Z^+_0$ such that
$$
\sup_{n\geq 0}|x(n)-y(n)|\leq\epsilon.
$$
The shadowing of Eq.~\eqref{ade} is closely related to its hyperbolicity. Recall that Eq.~\eqref{ade} is \emph{hyperbolic} if $\sigma(A)$ does not intersect the unit circle $|\lambda|=1$ in~$\mathbb C$, where $\sigma(A)$ denotes the spectrum of~$A$.

In a recent paper~\cite{BCDMP}, Bernardes~\emph{et al.} have studied various shadowing properties of Eq.~\eqref{ade}. Among others, they have shown that if Eq.~\eqref{ade} is hyperbolic, then it is shadowable on~$\mathbb Z^+_0$ (see \cite[Theorem~13]{BCDMP}). Moreover, if the coefficient~$A\in\mathcal L(X)$ in Eq.~\eqref{ade} is a compact operator, then the converse is also true. As usual, an operator $A\colon X\to X$ is \emph{compact} if for every bounded set $S\subset X$ the image $A(S)$ has compact closure in~$X$. Thus, we have the following theorem.
\begin{theorem}\label{motivthm1}
{\rm \cite[Theorem~15]{BCDMP}}
Let $A\in\mathcal L(X)$ be a compact operator. Then the following statements are equivalent.
\vskip5pt
\emph{(i)} Eq.~\eqref{ade} is shadowable on~$\mathbb Z^+_0$;
\vskip3pt

\vskip3pt

\emph{(ii)} Eq.~\eqref{ade} is hyperbolic.	

\end{theorem}

\begin{remark}\label{rem1}
As noted above, the implication (ii) $\Rightarrow$ (i) in Theorem~\ref{motivthm1} is true if we assume merely that $A\in\mathcal L(X)$. The compactness of $A\in\mathcal L(X)$ is important only for the validity of the converse implication (i) $\Rightarrow$ (ii) (see~\cite[Remark~14]{BCDMP}).
\end{remark}

Now let us consider the finite dimensional case $X=\mathbb C^d$, where $d$ is a positive integer and $\mathbb C^d$ denotes the $d$-dimensional space complex column vectors. Then, the space $\mathcal L(\mathbb C^d)$ can be identified with $\mathbb C^{d\times d}$, the space of $d\times d$ matrices with complex entries. Since linear operators between finite dimensional spaces are compact and the spectrum of a square matrix $A\in\mathbb C^{d\times d}$ consists of the roots of its characteristic equation
\begin{equation}\label{ache}
\det(\lambda E-A)=0,
\end{equation}
where $E\in\mathbb C^{d\times d}$ is the unit matrix, in this case, Theorem~\ref{motivthm1} can be reformulated as follows.

\begin{theorem}\label{motivthm2}
Let $A\in\mathbb C^{d\times d}$. Then the following statements are equivalent.
\vskip5pt
\emph{(i)} Eq.~\eqref{ade} is shadowable on~$\mathbb Z^+_0$;
\vskip3pt

\vskip3pt

\emph{(ii)} The characteristic equation~\eqref{ache} has no root on the unit circle $|\lambda|=1$.

\end{theorem}
Our aim in this paper is to extend Theorems~\ref{motivthm1} and~\ref{motivthm2} to more general classes of linear difference equations with delay. 

In Sec.~\ref{nefd}, we will generalize Theorem~1 to the nonautonomous linear difference equation with finite delays
\begin{equation}\label{nef}
x(n+1)=\sum_{j=0}^r A_j(n)x(n-j),	
\end{equation}
where $r\in\mathbb Z^+_0$ is the maximum delay and the coefficients $A_j(n)\in\mathcal L(X)$, $0\leq j\leq r$, 
$n\in\mathbb Z^+_0$, are compact linear operators which 
are uniformly bounded, i.e., there exists $K\geq1$ such that
\begin{equation}\label{unibdd}
	|A_j(n)|\leq K,\qquad n\in\mathbb Z^+_0,\quad 0\leq j\leq r.
\end{equation}
The main result of Sec.~\ref{nefd} is formulated in Theorem~\ref{mainresult1}, which may be viewed as a discrete analogue of our recent shadowing theorem for delay differential equations in~$\mathbb R^d$~\cite[Theorem~2.2]{BDP}. 
 It says that, under the above hypotheses, Eq.~\eqref{nef} is shadowable on~$\mathbb Z^+_0$ if and only if it has an exponential dichotomy which is
 a nonautonomous variant of hyperbolicity. To the best of our knowledge, this result is new even for ordinary difference equations ($r=0$). We note that in the particular case when $r=0$ and $X$ is finite-dimensional, a version of this result was established in~\cite{BD0}  (see~\cite[Corollary 2]{BD0} and~\cite[Proposition 4]{BD0}).

 In Sec.~\ref{veid}, we will extend Theorem~\ref{motivthm2} to the linear Volterra difference equation with infinite delay
 \begin{equation}\label{vde}
	x(n+1)=\sum_{j=-\infty}^n A(n-j)x(j),
\end{equation}
where $A\colon\mathbb Z^+_0\to\mathbb C^{d\times d}$ satisfies
\begin{equation}\label{gammacond}
\sum_{j=0}^\infty|A(j)|e^{\gamma j}<\infty\qquad\text{for some $\gamma>0$}.
\end{equation}
The \emph{characteristic equation} of Eq.~\eqref{vde} has the form
\begin{equation}\label{chareq}
\det\varDelta(\lambda)=0	,\qquad|\lambda|>e^{-\gamma},
\end{equation}
where
\begin{equation}\label{charfunc}
	\varDelta(\lambda)=\lambda E-\sum_{j=0}^\infty\lambda^{-j}A(j),
	\qquad|\lambda|>e^{-\gamma}.
\end{equation}
The main result of Sec.~\ref{veid}, Theorem~\ref{mainresult2}, says that in the natural (infinite dimensional) phase space~$\mathcal B_\gamma$ defined below Eq.~\eqref{vde} is shadowable on~$\mathbb Z^+_0$ if and only if its characteristic equation~\eqref{chareq} has no root on the unit circle $|\lambda|=1$.

The fact that nonautonomous linear delay difference equations, including~\eqref{nef} and~\eqref{vde}, are shadowable whenever they are hyperbolic follows from \cite[Theorem~1]{DP}. 
Therefore, in both cases~\eqref{nef} and~\eqref{vde}, we need to prove only the converse result.  
In the case of the nonautonomous equation with finite delays~\eqref{nef}, the proof follows similar lines as the proof our continuous time result~\cite[Theorem~2.2]{BDP} with nontrivial modifications because, instead of~$\mathbb R^d$, we consider Eq.~\eqref{nef} in a general infinite dimensional Banach space. It is based on the eventual compactness of the solution operator, combined with 
 an input-output technique~\cite{BDV} and Sch\"affer's result about the existence of regular covariant subspaces of linear difference equation in a Banach space~\cite{Schaffer}. A similar argument for the Volterra equation with infinite delay~\eqref{vde} does not apply since its solution operator is not eventually compact. In this case the proof will be based on the duality between Eq.~\eqref{vde} and its formal adjoint equation which has been established by Matsunaga~\emph{et.~al.}~\cite{MMNN}.

\section{Shadowing of nonautonomous linear difference equation with finite delays}\label{nefd}

In this section, we consider the shadowing of
the nonautonomous linear difference equation with finite delays~\eqref{nef}, where $A_j\colon\mathbb Z^+_0\to\mathcal L(X)$, $0\leq j\leq r$, satisfy condition~\eqref{unibdd}.
The \emph{phase space} for Eq.~\eqref{nef} is $(\mathcal B_r,\|\cdot\|)$, where $\mathcal B_r$ is the set of all functions $\phi\colon[-r,0]\cap\mathbb Z\to X$ and
\begin{equation*}
\|\phi\|=\max_{-r\leq\theta\leq0}|\phi(\theta)|,\qquad\phi\in\mathcal B_r.
\end{equation*}
Eq.~\eqref{nef} can be written equivalently in a form of a \emph{functional difference equation}
\begin{equation}\label{LDE}
x(n+1)=L_n (x_n) ,
\end{equation}
where the \emph{solution segment} $x_n\in\mathcal B_r$ is defined by 
\begin{equation*}
	x_n(\theta)=x(n+\theta),\qquad \theta \in[-r,0]\cap\mathbb Z,
	\end{equation*}
	 and $L_n\colon\mathcal B_r\to X$ is a bounded linear functional defined by
	 \begin{equation*}
	 	L_n(\phi)=\sum_{j=0}^r A_j(n)\phi(-j),\qquad \phi\in\mathcal B_r,\quad n\in\mathbb Z^+_0,\quad 0\leq j\leq r.
	 \end{equation*} 
In view of~\eqref{unibdd}, we have that
\begin{equation}\label{funibdd}
	\|L_n\|\leq M:=(r+1)K,\qquad n\in\mathbb Z^+_0.
\end{equation}	 
	 
Given $m\in\mathbb Z^+_0$ and $\phi\in\mathcal B_r$, there exists a unique function $x\colon\mathbb Z^+_{m-r}\to X$ satisfying Eq.~\eqref{nef} 
and the \emph{initial condition} $x(m+\theta)=\phi(\theta)$ for $\theta\in[-r,0]\cap\mathbb Z$. We shall call~$x$ the \emph{solution of Eq.~\eqref{nef} with initial value $x_m=\phi$}.
 By a \emph{solution of Eq.~\eqref{nef} on~$\mathbb Z_m^+$}, we mean a solution~$x$ with initial value $x_m=\phi$ for some~$\phi\in\mathcal B_r$.
 
 For each $n, m\in\mathbb Z^+_0$ with $n\geq m$, the \emph{solution operator} $T(n,m)\colon\mathcal B_r\to\mathcal B_r$ is defined by 
 $T(n,m)\phi=x_n$ for $\phi\in\mathcal B_r$, where $x$ is the unique solution of Eq.~\eqref{nef} with initial value $x_m=\phi$.
 It is easily seen that for all $n,k,m\in\mathbb Z^+_0$ with $n\geq k\geq m$,   \begin{gather}
 T(m,m)=I,\label{ev1}	\\
 T(n,m)=T(n,k)T(k,m),\label{ev2}\\
 \|T(n,m)\|\leq e^{\omega(n-m)},\label{ev3}
 \end{gather}
where $I$ denotes the identity operator on~$\mathcal B_r$ and $\omega=\log (M(1+r))$.

 Now we can introduce the definitions of shadowing and exponential dichotomy for Eq.~\eqref{nef} (equivalently, \eqref{LDE}).
\begin{definition}\label{shadow}
We say that Eq.~\eqref{nef} is \emph{shadowable on~$\mathbb Z^+_0$} if for each $\epsilon >0$ there exists $\delta>0$ such that for every function $y\colon \mathbb Z^+_{-r}\to X$ satisfying
\[
\sup_{n\geq 0}|y(n+1)-L_n(y_n)| \le \delta, 
\]
there exists a solution $x$ of~\eqref{nef} on~$\mathbb Z^+_0$ such that  
\[
\sup_{n\geq 0} \|x_n-y_n\| \le \epsilon.
\]
\end{definition}

\begin{definition}\label{ed}
We say that Eq.~\eqref{nef} admits an \emph{exponential dichotomy (on~$\mathbb Z^+_0$)} if there exist a sequence of projections $(P_n)_{n\in\mathbb Z^+_0}$ on~$\mathcal B_r$ and constants $D, \lambda >0$ with the following properties:
\begin{itemize}
\item for $n,m\in \mathbb Z^+_0$ with $n\geq m$, 
\begin{equation}\label{pro}
P_n T(n,m)=T(n,m)P_m,
\end{equation}
and $T(n,m)\rvert_{\operatorname{ker} P_m} \colon \operatorname{ker} P_m\to
\operatorname{ker} P_n$ is onto and invertible;
\item for $n,m\in \mathbb Z^+_0$ with $n\geq m$, 
\begin{equation}\label{eq: def est stable}
\|T(n,m)P_m\| \le De^{-\lambda (n-m)};
\end{equation}
\item for $n,m\in \mathbb Z^+_0$ with $n\leq m$,
\begin{equation}\label{eq: def est unstable}
\|T(n,m)Q_m\| \le De^{-\lambda (m-n)},
\end{equation}
where $Q_m=I-P_m$ and $T(n,m):=\left (T(m,n)\rvert_{\operatorname{ker} P_n} \right )^{-1}$.
\end{itemize}
\end{definition}

The main result of this section is the following shadowing theorem for Eq.~\eqref{nef}.
\begin{theorem}\label{mainresult1}
Suppose that the coefficients $A_j(n)\in\mathcal L(X)$, $n\in\mathbb Z^+_0$, $0\leq j\leq r$, of Eq.~\eqref{nef} are compact operators 
satisfying condition~\eqref{unibdd}. Then, the following statements are equivalent.
\vskip5pt
\emph{(i)} Eq.~\eqref{nef} is shadowable on~$\mathbb Z^+_0$;
\vskip3pt

\vskip3pt

\emph{(ii)} Eq.~\eqref{nef} admits an exponential dichotomy.

\end{theorem}

\begin{remark}\label{rem2}
It is well-known that the autonomous linear equation~\eqref{ade} admits an exponential dichotomy if and only if the spectrum $\sigma(A)$ does not intersect the unit circle $|\lambda|=1$. Therefore, Theorem~\ref{mainresult1} is a generalization of Theorem~1 to the nonautonomous delay difference equation~\eqref{nef}. Its conclusion is new even for ordinary difference equations ($r=0$).
\end{remark}

\begin{remark}\label{rem3}
The implication (ii) $\Rightarrow$ (i) in Theorem~\ref{mainresult1} is a consequence of \cite[Theorem~1]{DP}
with $f=0$, $c=0$ and $\mu=1$ which does not require the compactness of the coefficients. Thus, this implication is true even without the compactness assumption. However, for the validity of the converse implication (i) $\Rightarrow$ (ii)  
the compactness of the coefficient operators of Eq.~\eqref{nef} is essential (see Remark~\ref{rem1}).
\end{remark}

\begin{remark}
It follows from~\eqref{eq: def est stable} and~\eqref{eq: def est unstable} that if Eq.~\eqref{nef} admits an exponential dichotomy, then the solution operator $T(m, n)$ of Eq.~\eqref{nef} exhibits the (one-sided) domination property in the sense of~\cite[p.~2]{Quas}. In~\cite[Theorem 1.2]{Quas},  the authors have formulated sufficient conditions under which the solution operator associated with a nonautonomous difference equation (without delay)
\[
x_{n+1}=A_n x_n, \qquad n\in \mathbb Z,
\]
on an arbitrary Banach space $X$
exhibits the domination property. We stress that no compactness assumptions on the coefficients $A_n$, $n\in \mathbb Z$,
are assumed. These sufficient conditions are expressed in terms of the so-called \emph{uniform singular valued gap property} (see~\cite[(SVG)]{Quas}). For related results in the case of linear cocycles over topological dynamical systems, we refer to the works of Bochi and Gourmelon~\cite{BN} and Blumenthal and Morris~\cite{BM}, where the connection between this type of results and the Oseledets multiplicative ergodic theorem is discussed.
Our Theorem~\ref{mainresult1} provides a characterization of the more restrictive notion of uniform exponential dichotomy which is expressed in terms of the shadowing property instead of the singular values.
\end{remark}

\begin{proof}[Proof of Theorem~\ref{mainresult1}]
	
As noted 
in Remark~\ref{rem3},
 we need to prove only the implication (i) $\Rightarrow$ (ii). Suppose that Eq.~\eqref{nef} is shadowable on~$\mathbb Z^+_0$. We will show that it admits an exponential dichotomy. We split  the proof into several auxiliary results which we now briefly describe.
 
 In Claim~\ref{LEM1}, we show that the shadowing property implies the so-called Perron property which guarantees that for each bounded function $z\colon \mathbb Z_0^+\to X$ the nonhomogeneous equation~\eqref{adm} has at least one bounded solution $x\colon \mathbb Z_0^+\to X$. 
 
         The next four claims are preparatory results for the proof of the crucial Claim~\ref{covariance} which shows that the subspace $\mathcal S(0)$ of those initial functions in~$\mathcal B_r$ which generate bounded solutions is closed and complemented in~$\mathcal B_r$.
 Claims~\ref{INV} and~\ref{subcomplete} are rather simple observations, while Claim~\ref{ee} is a straightforward consequence of Claim~\ref{LEM1}. A more involved argument is needed for the proof of Claim~\ref{compcrit} which asserts that the solution operator $T(n, m)$ of Eq.~\eqref{nef} is a compact operator on $\mathcal B_r$ whenever $n\ge m+r+1$. The proof of Claim~\ref{covariance} follows directly from  Claims~\ref{INV}, \ref{ee}, \ref{subcomplete} and~\ref{compcrit} by applying an abstract result from~\cite{Schaffer} formulated in Lemma~\ref{schaef}. 
 
 As a consequence of Claim~\ref{covariance}, we are able to construct the unstable subspace $\mathcal U$ at time $n=0$ as a topological complement of~$\mathcal S(0)$  (see~\eqref{SPLIT}), and we can revisit the Perron property established in Claim~\ref{LEM1}. More precisely, in Claim~\ref{LEM2}, we show that for each bounded $z\colon \Z_0^+\to X$ there exists a \emph{unique} bounded solution $x\colon \Z_0^+\to X$ of Eq.~\eqref{adm} with $x_0\in \mathcal U$. Moreover, the supremum norm of~$x$ can be controlled by the supremum norm of~$z$ (see~\eqref{cgt}).  

In the next step, we construct the unstable subspace~$\mathcal U(n)$ at each time $n\in \mathbb Z^+$. In Claims~\ref{lem: invertibility} and~\ref{splitting}, we prove that $T(n, m)\rvert_{\mathcal U(m)}\colon \mathcal U(m)\to \mathcal U(n)$ is an isomorphism whenever $n\ge m$ and the phase space~$\mathcal B_r$ splits into stable and unstable subspaces at each moment $n\in \mathbb Z^+$. Note that both claims are consequences of Claim~\ref{LEM2}.

 The desired exponential estimates along the stable and unstable directions are obtained in Claims~\ref{exp est stable} and~\ref{exp est unstable}, respectively. As a preparation for the proofs of these results,  in Claims~\ref{C} and~\ref{C-U}, we show that the dynamics along the stable and unstable directions is uniformly bounded forward and backward in time, respectively.
 These results also rely on Claim~\ref{LEM2}. 

 Finally, in Claim~\ref{Claim: projectionsbounded}, we prove that there is a uniform bound for the norms of the projections onto the stable subspaces along the unstable ones. 
 
 We now proceed with the details.

\begin{Claim}\label{LEM1}
Eq.~\eqref{nef} has the following \emph{Perron type property}:
for each bounded function $z\colon \mathbb Z^+_0\to X$, 
there exists a bounded function $x\colon \mathbb Z^+_{-r}\to X$ which satisfies
\begin{equation}\label{adm}
x(n+1)=L_n(x_n)+z(n), \qquad n\in \mathbb Z^+_0.
\end{equation}
\end{Claim}
\begin{proof}[Proof of Claim~\ref{LEM1}]
Let $z\colon \mathbb Z^+_0\to X$ be an arbitrary bounded function. If $z(n)=0$ for every $n\in \mathbb Z^+_0$, then~\eqref{adm} is trivially satisfied with $x(n)=0$ for every $n\in \mathbb Z^+_{-r}$. Now suppose that $z(n)\neq0$ for some $n\in \mathbb Z^+_0$ so that $\|z\|_\infty:=\sup_{n\in \mathbb Z^+_0}|z(n)|>0$.
Choose a constant $\delta >0$ corresponding to the choice of $\epsilon=1$ in Definition~\ref{shadow}. Take an arbitrary solution $y\colon \mathbb Z^+_{-r}\to X$ of the nonhomogeneous equation
\[
y(n+1)=L_n(y_n)+\frac{\delta}{\|z\|_\infty}z(n), \qquad n\in \mathbb Z^+_0.
\]
(The unique solution~$y$ with initial value $y_0=0$ is sufficient for our purposes.) Since
\[
\sup_{n\geq 0}|y(n+1)-L_n(y_n)| \le \delta,
\]
according to Definition~\ref{shadow}, there exists a solution~$\tilde x$ 
 of Eq.~\eqref{LDE} on $\mathbb Z^+_0$ such that
\[
\sup_{n\geq{-r}}|\tilde x(n)-y(n)|=\sup_{n\geq 0}\|\tilde x_n-y_n\| \le 1.
\]
Define a function $x\colon \mathbb Z^+_{-r}\to X$ by 
\[
x(n):=\frac{\|z\|_\infty}{\delta}(\,y(n)-\tilde x(n)\,), \qquad n\in \mathbb Z^+_{-r}.
\]
It can be easily verified that $x$ satisfies~\eqref{adm} and 
\[
\sup_{n\in \mathbb Z^+_{-r}}|x(n)|\le \frac{\|z\|_\infty}{\delta}<\infty.
\]
The proof of the claim is complete.
\end{proof}

For each $m\in \mathbb Z^+_0$, define
\[
\mathcal S(m)=\bigl \{ \phi\in \mathcal B_r: \sup_{n\ge m}\|T(n,m)\phi \|<\infty \bigr \}.
\]
Clearly, $\mathcal S(m)$ is a subspace of $\mathcal B_r$ which will be called the \emph{stable subspace of Eq.~\eqref{LDE} at~$m\in \mathbb Z^+_0$}.

\begin{Claim}\label{INV}
For each $n,m\in \mathbb Z^+_0$ with $n\ge m $, we have that 
\[
[\,T(n,m)\,]^{-1}(\mathcal S(n))=\mathcal S(m).
\]
\end{Claim}

\begin{proof}[Proof of Claim~\ref{INV}]
Let $n$ and $m$ be as in the statement. If $\phi \in \mathcal S(m)$, then (see~\eqref{ev2})
\[
\sup_{k\geq n}\|T(k,n)T(n, m)\phi \|=\sup_{k\geq n}\|T(k,m)\phi \|\le \sup_{k\geq m}\|T(k,m)\phi \|<\infty. 
\]
This shows that $T(n,m)\phi \in \mathcal S(n)$ and hence $\phi \in [\,T(n,m)\,]^{-1}(\mathcal S(n))$.

Now suppose that $\phi \in [\,T(n,m)\,]^{-1}(\mathcal S(n))$. Then, $T(n,m)\phi \in \mathcal S(n)$, which implies that
\begin{equation*}
\sup_{k\geq n}\|T(k,m)\phi\|=\sup_{k\geq n}\|T(k,n)T(n,m)\phi\|<\infty.
\end{equation*}
Hence
\begin{equation*}
\sup_{k\geq m}\|T(k,m)\phi\|\leq \max_{m\leq k\leq n-1}\|T(k,m)\phi\|+\sup_{k\geq n}\|T(k,m)\phi\|<\infty.
\end{equation*}
Thus, $\phi \in \mathcal S(m)$.
\end{proof}

\begin{Claim}\label{ee}
For $n,m\in \mathbb Z^+_0$ with $n\ge m$, we have the algebraic sum decomposition 
\begin{equation}\label{SPL}
\mathcal B_r=T(n,m)\mathcal B_r+\mathcal S(n).
\end{equation}
\end{Claim}

\begin{proof}[Proof of Claim~\ref{ee}]
It is sufficient to prove the claim for $m=0$. Indeed, assuming that the desired conclusion holds for $m=0$, we now fix an arbitrary $m\geq 1$. Then, for every $n\geq m$ and $\phi \in \mathcal B_r$, there exist $\phi_1\in \mathcal B_r$ and $\phi_2\in \mathcal S(n)$ such that $\phi=T(n,0)\phi_1+\phi_2$. Hence,
\[
\phi=T(n,0)\phi_1+\phi_2=T(n,m)T(m,0)\phi_1+\phi_2\in T(n,m)\mathcal B_r+\mathcal S(n).
\]
Thus,~\eqref{SPL} holds.
Therefore, from now on, we suppose that $m=0$. Evidently, for every $\phi\in \mathcal B_r$, 
\[
\phi=T(0,0)\phi+0\in T(0,0)\mathcal B_r+\mathcal S(0).
\]
Thus,~\eqref{SPL} holds for $n=m=0$. Now suppose that $n>0$ and let
 $\phi \in \mathcal B_r$ be arbitrary.  
 Define $v\colon \mathbb Z^+_{n-r}\to X$ by 
\[
v(k)=\begin{cases}
    \phi(k-n) &\quad\text{for $n-r\leq k\leq n$},\\
    0 & \quad \text{for $k\geq n+1$}
\end{cases}
\]
so that $v_n=\phi$
and $z\colon\mathbb Z^+_0\to X$ by
\[
z(k)=\begin{cases}
0&\quad\text{for $0\leq k\leq n-1$},\\
v(k+1)-L_k(v_k) &\quad\text{for $k\geq n$}.
\end{cases}
\]
Clearly, $v$ is bounded on~$\mathbb Z^+_{n-r}$. From this and~\eqref{funibdd}, we find that $\sup_{k\geq 0} |z(k)|<\infty$.
By Claim~\ref{LEM1}, there exists a function $x\colon \mathbb Z^+_{-r} \to X$ such that $\sup_{k\geq-r}|x(k)|<\infty$ and~\eqref{adm} holds. Moreover, it follows from the definition of~$z$ that 
\[
v(k+1)=L_k(v_k)+z(k), \qquad k\geq n.
\]
From this and~\eqref{adm}, we conclude that $x-v$ is a solution of Eq.~\eqref{LDE} on $\mathbb Z^+_n$ and thus
\begin{equation*}\label{g}
x_k-v_k=T(k,n)(x_n-v_n)=T(k,n)(x_n-\phi), \qquad k\geq n.
\end{equation*}
Since both~$x$ and~$v$ are bounded on~$\mathbb Z^+_{n-r}$,  this implies that $x_n-\phi \in \mathcal S(n)$. On the other hand, since $z(k)=0$ for $0\leq k<n$, we have that $x_n=T(n, 0)x_0$. This implies that
\[
\phi=x_n+(\phi-x_n)=T(n,0)x_0+(\phi-x_n)\in T(n,0)\mathcal B_r+ \mathcal S(n).
\]
Since $\phi \in \mathcal B_r$ was arbitrary, we conclude that~\eqref{SPL} holds for $m=0$.
\end{proof}

\begin{Claim}\label{subcomplete}
For each $m\in \mathbb Z^+_0$, $\mathcal S(m)$ is the image of a  Banach space under the action of a bounded linear operator.
\end{Claim}

\begin{proof}[Proof of Claim~\ref{subcomplete}]
Fix $m\in \mathbb Z^+_0$ and let $\mathcal X$ denote the Banach space of all bounded functions $x\colon \mathbb Z^+_{m-r} \to X$ equipped with the supremum norm,
\[
\|x\|_{\mathcal X}:=\sup_{k\geq m-r}|x(k)|<\infty,\qquad x\in\mathcal X.
\]
Let $\mathcal X'$ denote the set of all $x\in \mathcal X$ which are solutions of~\eqref{LDE} on $\mathbb Z^+_0$. We will show that $\mathcal X'$ is a closed subspace of $\mathcal X$.  To this end, let $(x^j)_{j\in \mathbb Z^+_0}$ be a sequence in $\mathcal X'$ such that $x^j \to y$ in $\mathcal X$ as $j\to\infty$ for some $y\in\mathcal X$. Then, $x^j(k)\to y(k)$ as $j\to\infty$ for every $k\geq m-r$. Moreover, for each $n\geq m$, we have that
$x^j_n\to y_n$ in~$\mathcal B_r$ as $j\to\infty$. From this, by letting  $j\to\infty$ in the equation
\begin{equation*}
 x^j(n+1)=L_n(x^j_n),\qquad n\geq m,   
\end{equation*}
and using the continuity of the coefficients~$L_n$, we conclude that
\begin{equation*}
y(n+1)=L_n(y_n),\qquad n\geq m.  
\end{equation*}
Thus, $y$ is a solution of Eq.~\eqref{LDE} on $\mathbb Z^+_m$ and hence $y\in\mathcal X'$.
This shows that $\mathcal X'$ is a closed subspace of $\mathcal X$ and hence it is a Banach space. Now define $\Phi \colon \mathcal X' \to \mathcal B_r$ by $\Phi(x)=x_m$ for $x\in\mathcal X'$. Clearly, $\Phi$ is a bounded linear operator with 
$\| \Phi \| \le 1$ and $\Phi(\mathcal X')=\mathcal S(m)$.
\end{proof}

\begin{Claim}\label{compcrit} For $n, m\in \mathbb Z^+_0$ with $n\geq m+r+1$,
the solution operator
$T(n,m)\colon \mathcal B_r\to \mathcal B_r$ is compact. 
\end{Claim}

\begin{proof}[Proof of Claim~\ref{compcrit}] Suppose that $m\in\mathbb Z^+_0$ and $\phi\in \mathcal B_r$. Let $x$ denote the unique solution of~\eqref{LDE} with initial value $x_m=\phi$. From~\eqref{LDE} and~\eqref{funibdd}, we obtain for $n\geq m$,
\begin{equation*}
    |x(n+1)|\leq\|L_n\|\|x_n\|\leq M\|x_n\|,
\end{equation*}
and hence
\begin{equation*}
    \|x_{n+1}\|\leq|x(n+1)|+\|x_n\|\leq(1+M)\|x_n\|.
\end{equation*}
Since $\|x_m\|=\|\phi\|$, this implies by induction on~$n$ that
\begin{equation}\label{expest}
 \|x_n\|\leq(1+M)^{n-m}\|\phi\|\qquad\text{for $n\geq m$}.   
\end{equation}
Clearly, $\iota(\phi)=(\phi(-r),\phi(-r+1),\dots,\phi(0))$ is an isometric isomorphism between the phase space~$\mathcal B_r$ and the $(r+1)$-fold product space~$X^{r+1}$ endowed with the maximum norm, $\|x\|=\max_{1\leq j\leq r+1}|x_j|$ for $x=(x_1,x_2,\dots,x_{r+1})\in X^{r+1}$. With this identification, we have that
\begin{equation*}
T(m+r+1,m)\phi=x_{m+r+1}=(x(m+1),x(m+2),\dots,x(m+r+1)),
\end{equation*}
which, together with Eq.~\eqref{LDE}, implies that
\begin{equation}\label{soloprep}
T(m+r+1,m)\phi=\bigl(L_m(x_m),L_{m+1}(x_{m+1}),\dots,L_{m+r}(x_{m+r})\bigr).
\end{equation}
Let $S$ be an arbitrary bounded subset of~$\mathcal B_r$. Then, there exists $\rho>0$ such that $\|\phi\|\leq\rho$ for all $\phi\in S$. From this and~\eqref{expest}, we obtain that 
\begin{equation*}
\|x_{m+j}\|\leq(1+M)^{r}\rho\qquad\text{whenever $\phi\in S$ and $0\leq j\leq r$}.
\end{equation*}
Therefore, if $\phi\in S$, then the segments $x_m, x_{m+1},\dots,x_{m+r}$ of the corresponding solution~$x$ of~\eqref{LDE} with initial value $x_m=\phi$ belong to the closed ball of radius $(1+M)^{r}\rho$ around zero in~$\mathcal B_r$, which will be denoted by~$D$. From this and~\eqref{soloprep}, we conclude that
\begin{equation}\label{solopincl}
T(m+r+1,m)(S)\subset C:=\overline{L_m(D)}\times\overline{L_{m+1}(D)}\times\dots\times\overline{L_{m+r}(D)}.
\end{equation}
Since $L_m,L_{m+1},\dots,L_{m+r}$ are compact operators and~$D$ is a bounded set, the closures of the image sets
$L_m(D), L_{m+1}(D),\dots,L_{m+r}(D)$ are compact subsets of~$X$. Therefore, $C$ is a product of $r+1$ compact subsets of~$X$ which is a compact subset of the product space~$X^{r+1}$. In view of~\eqref{solopincl}, $T(m+r+1,m)(S)$
is a subset of the compact set~$C\subset X^{r+1}$ and hence it is relatively compact in~$X^{r+1}$. Since $S$ was an arbitrary bounded subset of~$\mathcal B_r$, this proves that $T(m+r+1,m)\colon \mathcal B_r\to \mathcal B_r$ is a compact operator. Finally, if $n>m+r+1$, then
\begin{equation*}
T(n,m)=T(n,m+r+1)T(m+r+1,m)   
\end{equation*}
 is a product of the bounded linear operator $T(n,m+r+1)$ and the compact operator $T(m+r+1,m)$ and hence it is compact (see, e.g., \cite{Taylor}).
\end{proof}

\begin{Claim}\label{covariance}
The stable subspace~$\mathcal S(0)$ of Eq.~\eqref{LDE} is closed and has finite codimension in~$\mathcal B_r$.
\end{Claim}

Before giving a proof of Claim~\ref{covariance}, let us recall the following notions~\cite{Schaffer}. Let~$B$ be a Banach space.
A subspace $S$ of~$B$ is called \emph{subcomplete in~$B$} if there exist a Banach space~$Z$ and a bounded linear operator $\Phi\colon Z\to B$ such that $\Phi(Z)=S$

Let $\mathcal A\colon\mathbb Z^+_0\to\mathcal L(B)$ be an operator-valued map. For $n,m\in \mathbb Z^+_0$ with $n\geq m$, define the corresponding \emph{transition operator} $U(n,m)\colon B\to B$  by
\[
U(n,m)=\mathcal A(n-1)\mathcal A(n-2)\cdots \mathcal A(m)\qquad\text{for $n>m$}
\]
and $U(n,n)=I$ for $n\in\mathbb Z^+_0$, where $I$ denotes the identity operator on $B$. A sequence $Y=(Y(n))_{n\in\mathbb Z^+_0}$ of subspaces in~$B$ is called a \emph{covariant sequence for~$\mathcal A$} if 
\[
[\,\mathcal A(n)\,]^{-1}(Y(n+1))=Y(n) \qquad\text{for all $n\in \mathbb Z^+_0$}.
\]
A covariant sequence~$Y=(Y(n))_{n\in\mathbb Z^+_0}$ for~$\mathcal A$ is called \emph{algebraically regular} if
\[
U(n,0)B+Y(n)=B\qquad\text{for each $n\in \mathbb Z^+_0$}.
\]
Finally, a covariant sequence~$Y=(Y(n))_{n\in\mathbb Z^+_0}$ for~$\mathcal A$ is called \emph{subcomplete} if the subspace $Y(n)$ is subcomplete in~$B$ for all $n\in\mathbb Z^+_0$.

The proof of Claim~\ref{covariance} will be based on the following result due to Sch\"affer~\cite{Schaffer}.

\begin{lemma}\label{schaef}
{\rm (\cite[Lemma~3.4]{Schaffer})}
Let $B$ be a Banach space and $\mathcal A\colon\mathbb Z^+_0\to\mathcal L(B)$. Suppose that 
$Y=(Y(n))_{n\in\mathbb Z^+}$ is a subcomplete algebraically regular covariant sequence for~$\mathcal A$. If the transition operator $U(n,m)\colon B\to B$ is compact for some $n,m\in\mathbb Z^+_0$ with $n>m$, then the subspaces $Y(n)$, $n\in\mathbb Z^+_0$, are closed and have constant finite codimension in~$B$.
\end{lemma}

Now we can give a proof of Claim~\ref{covariance}.
\begin{proof}[Proof of Claim~\ref{covariance}]
Claims~\ref{INV}, \ref{ee} and~\ref{subcomplete} guarantee that the stable subspaces $Y(n):=\mathcal S(n)\subset \mathcal B_r$ of Eq.~\eqref{LDE} 
form a subcomplete algebraically regular covariant sequence for $\mathcal A\colon\mathbb Z^+\to\mathcal L(\mathcal B_r)$ defined by
\[
\mathcal A(n):=T(n+r+1, n), \qquad n\in \mathbb Z^+_0.
\]
According to Claim~\ref{compcrit}, the associated transition operator $U(n+1,n)=\mathcal A(n)$ is compact. By the application of Lemma~\ref{schaef}, we conclude that $Y(0)=\mathcal S(0)$ is closed and has finite codimension in~$\mathcal B_r$.
\end{proof}

By Claim~\ref{covariance}, the stable subspace~$\mathcal S(0)$ is closed and has finite codimension in~$\mathcal B_r$. This implies that~$\mathcal S(0)$ is complemented in~$\mathcal B_r$ (see, e.g., \cite[Lemma~4.21, p.~106]{Rud}). More precisely,  there exists a subspace~$\mathcal U$ of~$\mathcal B_r$ such that $\dim\mathcal U=\operatorname{codim}\mathcal S(0)<\infty$ and
\begin{equation}\label{SPLIT}
\mathcal B_r=\mathcal S(0)\oplus \mathcal U.
\end{equation}
\begin{Claim}\label{LEM2}
For each bounded function $z\colon \mathbb Z^+_0\to X$, 
there exists a unique bounded function $x\colon \mathbb Z^+_{-r} \to X$ with $x_0\in \mathcal U$ which satisfies ~\eqref{adm}. Moreover, there exists a constant $C>0$, independent of $z$, such that 
\begin{equation}\label{cgt}
\sup_{n\geq-r}|x(n)| \le C\sup_{n\geq0}|z(n)|.
\end{equation}
\end{Claim}

\begin{proof}[Proof of Claim~\ref{LEM2}]
By Claim~\ref{LEM1}, there exists a bounded function $\tilde x\colon\mathbb Z^+_{-r}\to X$ that satisfies
\[
\tilde x(n+1)=L_n(\tilde x_n)+z(n), \qquad n\in \mathbb Z^+_0.
\]
On the other hand, \eqref{SPLIT} implies the existence of~$\phi_1\in \mathcal S(0)$ and $\phi_2\in \mathcal U$ such that 
\[
\tilde x_0=\phi_1+\phi_2.
\]
Define $x\colon\mathbb Z^+_{-r} \to X$ by 
\[
x(n)=\tilde x(n)-y(n), \qquad n\geq-r,
\]
where $y$ is a solution of Eq.~\eqref{LDE} with initial value $y_0=\phi_1$. Since $y_0=\phi_1\in\mathcal S(0)$, we have that $\sup_{n\geq-r} |y(n)|<\infty$. Then, $x$ satisfies~\eqref{adm}, $x_0=\tilde x_0-\phi_1=\phi_2\in \mathcal U$ and $\sup_{n\geq-r}|x(n)|<\infty$. 
We claim that $x$ with the desired properties is unique. Indeed, if $\bar x$ is an arbitrary function with the desired properties, then  $x_0-\bar x_0\in \mathcal U\cap \mathcal S(0)=\{0\}$. Thus, $x_0=\bar x_0$ and hence $x=\bar x$ identically on $\mathbb Z^+_{-r}$.

Finally, we show the existence of a constant $C>0$ such that~\eqref{cgt} holds. 
Let $\mathcal X_0$ and $\mathcal X_{-r}$ denote the Banach space of all bounded $X$-valued functions defined on $\mathbb Z^+_0$ and $\mathbb Z^+_{-r}$, respectively, equipped with the supremum norm.
For $z\in\mathcal X_0$, define 
$\mathcal F(z)=x$, where $x$ is the unique bounded solution of the nonhomogeneous equation~\eqref{adm} with $x_0\in\mathcal U$.
(The existence and uniqueness of~$x$ is guaranteed by the first part of the proof.) Evidently, $\mathcal F(z)=x\in\mathcal X_{-r}$ for $z\in\mathcal X_0$ and $\mathcal F \colon \mathcal X_0\to \mathcal X_{-r}$ is a linear operator.
We will now observe that $\mathcal F$ is a closed operator. Indeed, let  $(z^k)_{k\in \mathbb Z^+}$ be a sequence in $\mathcal X_0$ such that $z^k\to z$ for some $z\in\mathcal X_0$ and $x^k:=\mathcal F (z^k) \to x$ for some $x\in\mathcal X_{-r}$. Then, letting $k\to +\infty$ in
\[x^k(n+1)=L_n(x_n^k)+z^k(n) \]
for each fixed $n$, we get that
\[x(n+1)=L_n(x_n)+z(n), \qquad n\in \mathbb Z^+_0.\]
That is, $x$ satisfies~\eqref{adm}. Now, since $x^k_0\in\mathcal U$ for $k\in\mathbb Z^+_0$ and $\mathcal U$ is a finite-dimensional and hence a closed subset of~$\mathcal B_r$, we have that $x_0=\lim_{k\to\infty}x^k_0\in\mathcal U$. Therefore, $x\in\mathcal X_{-r}$ is a bounded function satisfying~\eqref{adm} with $x_0\in\mathcal U$. Hence $\mathcal F(z)=x$ which shows that 
$\mathcal F \colon \mathcal X_0\to \mathcal X_{-r}$
is a closed operator. According to the Closed Graph Theorem (see, e.g., \cite[Theorem~4.2-I, p.~181]{Taylor}), $\mathcal F$ is bounded, which implies that~\eqref{cgt} holds with $C=\|\mathcal F\|$, the operator norm of~$\mathcal F$.
\end{proof}

For $n\in \mathbb Z^+$, define
\[ \mathcal U(n)=T(n,0)\mathcal U
\]
so that $\mathcal U(0)=\mathcal U$.
It is easily seen that 
\begin{equation}\label{eq: invariance S U}
T(n,m)\mathcal S(m)\subset \mathcal S(n) \quad \text{and} \quad T(n,m)\mathcal U(m)=\mathcal U(n)
\end{equation}
whenever $n,m\in\mathbb Z^+_0$ with $n\geq m$.

\begin{Claim}\label{lem: invertibility}
For $n,m\in \mathbb Z^+_0$ with $n\geq m$, $T(n,m)\rvert_{\mathcal U(m)} \colon \mathcal U(m)\to \mathcal U(n)$ is invertible.
\end{Claim}
\begin{proof}[Proof of Claim~\ref{lem: invertibility}]
In view of~\eqref{eq: invariance S U}, we only need to show that the operator above is injective. Let $n,m\in \mathbb Z^+_0$ with $n\geq m$ and $\phi \in \mathcal U(m)$ be such that $T(n,m)\phi=0$. Since $\phi \in \mathcal U(m)$, there exists $\bar{\phi}\in \mathcal U(0)=\mathcal U$ such that $\phi=T(m,0)\bar{\phi}$. Let $x\colon \mathbb Z^+_{-r}\to X$ be the solution of~\eqref{LDE} with initial value $x_0=\bar{\phi}$. Since $x(k)=0$ for all sufficiently large $k\in \mathbb Z^+$, we have that $\sup_{k\geq-r}|x(k)|<\infty$. It follows from the uniqueness in Claim~\ref{LEM2}, applied for $z\equiv 0$, that $x\equiv 0$. This implies that $\bar{\phi}=\phi= 0$.
\end{proof}
\begin{Claim}\label{splitting}
For each $n\in \mathbb Z^+_0$, $\mathcal B_r$ can be decomposed into the direct sum
\begin{equation}\label{SPLImainresult1} 
\mathcal B_r=\mathcal S(n)\oplus \mathcal U(n).
\end{equation}
\end{Claim}
\begin{proof}[Proof of Claim~\ref{splitting}]
Since $\mathcal U(0)=\mathcal U$, for $n=0$, the decomposition~\eqref{SPLImainresult1} follows immediately from~\eqref{SPLIT}. Now suppose that
$n\geq 1$ and let $\phi \in \mathcal B_r$ be arbitrary. Let $v\colon \mathbb Z^+_{n-r}\to X$ and $z\colon \mathbb Z^+_0\to X$ be as in the proof of Claim~\ref{ee}. Since $\sup_{k\geq0} |z(k)|<\infty$, by Claim~\ref{LEM2} there exists a  unique function $x\colon\mathbb Z^+_{-r} \to X$ such that $x_0\in \mathcal U$, $\sup_{k\geq{-r}}|x(k)|<\infty$ and~\eqref{adm} holds. By the same reasoning as in the proof of Claim~\ref{ee}, we have that $x_n-\phi\in \mathcal S(n)$. Moreover, $x_n=T(n,0)x_0\in \mathcal U(n)$. Consequently,
\[
\phi=(\phi-x_n)+x_n\in \mathcal S(n)+\mathcal U(n).
\]
Suppose now that $\phi \in \mathcal S(n)\cap \mathcal U(n)$. Then, there exists $\bar{\phi}\in \mathcal U$ such that $\phi=T(n,0)\bar{\phi}$. Consider the unique solution $x\colon \mathbb Z^+_{-r}\to X$ of Eq.~\eqref{LDE} with $x_0=\bar{\phi}$. Then, $x$ satisfies~\eqref{adm} with $z\equiv 0$, $x_0=\bar\phi\in \mathcal U$ and $\sup_{k\geq{-r}}|x(k)|<\infty$. By the uniqueness in Claim~\ref{LEM2}, we  conclude that $x\equiv 0$. Therefore, $\bar\phi=0$ which implies that $\phi=0$.
The proof of the Claim is completed.
\end{proof}

\begin{Claim}\label{C}
There exists $Q>0$ such that 
\[
\|T(n,m)\phi \| \le Q\|\phi \|, 
\]
for every $n,m\in \mathbb Z^+_0$ with $n\ge m$ and $\phi \in \mathcal S(m)$.
\end{Claim}
\begin{proof}[Proof of Claim~\ref{C}]
Fix $m\in \mathbb Z^+_0$ and $\phi \in \mathcal S(m)$. Let $u\colon \mathbb Z^+_{m-r} \to X$ be the solution of Eq.~\eqref{LDE} with initial value $u_m=\phi$. 
Define $x\colon \mathbb Z^+_{-r} \to X$ and $z\colon\mathbb Z^+_0\to X$ by 
\[
x(k)=\begin{cases}
 u(k)&\text{ for $k\geq{m+1}$;} \\
0 &\text{ for $-r\leq k\leq m$}
\end{cases}
\]
and
\[
z(k)=x(k+1)-L_k(x_k),\qquad k\geq0,
\]
respectively.
Evidently, \eqref{adm} is satisfied, and since $\phi \in \mathcal S(m)$, we have that $\sup_{k\geq{-r}}|x(k)|<\infty$.
Moreover, $x_0=0\in \mathcal U$. Furthermore, $z(k)=0$ for $0\leq k\leq m-1$ 
 and $k\geq m+r+1$.
 Thus, using~\eqref{funibdd}, ~\eqref{ev3} and the fact that $\|x_k\|\leq \|u_k\|$, we find that 
\[
\begin{split}
\sup_{k\geq0}|z(k)|&=\sup_{m\leq k \leq m+r}|z(k)|\\
&\le \sup_{m\leq k \leq m+r }|x(k+1)|+\sup_{m\leq k \leq m+r }|L_kx_k|\\
&\leq \sup_{m\leq k \leq m+r}|u(k+1)|+M\sup_{m\leq k \leq m+r }\|x_k\|\\
&\leq \sup_{m\leq k \leq m+r}|L_k(u_{k})|+M\sup_{m\leq k \leq m+r}\|u_k\|\\
&\leq 2M\sup_{m\leq k \leq m+r }\|u_k\|\\
&= 2 M\sup_{m\leq k \leq m+r}\|T(k,m)\phi\|\\
&\leq 2Me^{\omega r}\| \phi\|.
\end{split}
\]
From the last inequality and conclusion~\eqref{cgt} of Claim~\ref{LEM2}, we conclude that 
\[
\sup_{k\geq{m+1}}|u(k)| \le \sup_{k\geq{-r}}|x(k)|  \le C\sup_{k\geq 0}|z(k)| \le 2CMe^{\omega r}\| \phi\|.
\]
Hence, 
\[
\|T(n,m)\phi \|=\|u_n\| \le 2CMe^{\omega r}\| \phi\|, \qquad n\ge m+r+1.
\]
Since~\eqref{funibdd} implies that 
\[
\|T(n,m)\phi \| \le e^{\omega r}\|\phi\|, \qquad m\leq n\leq m+r,
\]
the conclusion of the claim holds with
\[
Q:=\max \{\,e^{\omega r}, 2CMe^{\omega r}\,\}>0.
\]
\end{proof}

\begin{Claim}\label{exp est stable}
There exist $D, \lambda >0$ such that 
\[
\|T(n,m)\phi \| \le De^{-\lambda (n-m)}\|\phi\|,
\]
for every $n,m\in \mathbb Z^+_0$ with $n\geq m$ and $\phi \in \mathcal S(m)$.
\end{Claim}

\begin{proof}[Proof of Claim~\ref{exp est stable}]
We claim that if
\begin{equation}\label{Nchoice}
N>eCMQ^2(r+1)+r+1
\end{equation}
with~$Q$ as in Claim~\ref{C},
then for every $m\in \mathbb Z^+_0$ and $\phi\in\mathcal S(m)$,
\begin{equation}\label{e}
\|T(n,m)\phi \| \le \frac 1 e \|\phi \|\qquad\text{for $n\geq m+N$}.
\end{equation}
Suppose, for the sake of contradiction, that \eqref{Nchoice} holds and there exist $m\in \mathbb Z^+_0$ and $\phi\in\mathcal S(m)$ such that
\begin{equation}\label{low}
\|T(n,m)\phi \|>\frac 1 e \|\phi \|\qquad\text{for some $n\geq m+N$}.
\end{equation}
Fix $n\geq m+N$ such that the first inequality in~\eqref{low} holds and let $u$ denote the solution of the homogeneous equation Eq.~\eqref{LDE} with initial value $u_m=\phi$ so that $u_{n}=T(n,m)\phi$. Therefore, the first inequality in~\eqref{low} can be written as
\begin{equation}\label{lowup}
\|u_n\|>\frac{1}{e}\|\phi\|. 
\end{equation}
In view of~\eqref{eq: invariance S U}, $\phi\in\mathcal S(m)$ implies that $u_j=T(j,m)\phi\in\mathcal S(j)$ for $j\geq m$. Therefore, by Claim~\ref{C}, we have that
\begin{equation*}
\|u_{n}\|=\|T(n,j)u_j\|\leq Q\|u_j\|\qquad \text{whenever $m\leq j\leq n$}.
\end{equation*}
From the last inequality and~\eqref{lowup}, we find that $\|u_j\|>0$ whenever $ m\leq j\leq n$.
This, together with the fact that $n\ge m+N> m+r+1$ and hence $n-r-1>m$, implies that we can define a function $x\colon\mathbb Z^+_{-r}\to X$  by
\[
x(k)=\begin{cases}
\chi(k)u(k)& \text{for $k\geq m$,}\\
0 & \text{for $-r\leq k\leq m-1$,}
\end{cases}
\]
where
\[
 \chi(k)=\begin{cases}
 0\qquad&\text{for $-r\leq k\leq m$},\\[2pt]
 \displaystyle\sum_{j=m}^{k-1} \|u_j\|^{-1}
 \qquad&\text{for $m+1\leq k\leq n-r-1$}, \\[10pt]
 \displaystyle \sum_{j=m}^{n-r-1} \|u_j\|^{-1}
 \qquad&\text{for $k \geq n-r$}.
 \end{cases}
 \]
Evidently, $x$ satisfies the nonhomogeneous equation
\begin{equation}\label{nonhomog}
x(k+1)=L_k(x_k)+z(k),\qquad k\in \mathbb Z^+_0,
\end{equation}
where $z\colon \mathbb Z^+_0\rightarrow X$ is defined by
\[
z(k)=x(k+1)-L_k(x_k), \qquad  k\in \mathbb Z^+_0.
\]
Since $u_m=\phi\in\mathcal S(m)$ implies that $u$ is bounded on~$\mathbb Z^+_{m-r}$ and  $0\leq\chi(k)\leq \chi(n-r)$ for $k\geq{-r}$, it follows that $x$ is bounded on~$\mathbb Z^+_{-r}$. From this and~\eqref{funibdd}, we obtain that $z$ is also bounded on~$\mathbb Z^+_0$. Since $x_0=0\in \mathcal{U}$, by Claim~\ref{LEM2}, we have that
\begin{equation}\label{eq: est x z}
\sup_{k\geq{-r}}|x(k)|\leq C\sup_{k\geq0}|z(k)|.
\end{equation}
Our objective now is to estimate the norm of $z(k)$. Since~$x(k)=0$ for $-r\leq k\leq m$, we have that $z(k)=0$ for $0\leq k\leq m-1$. The function $\chi$ is constant on~$\mathbb Z^+_{n-r}$, therefore $x$ is a constant multiple of the solution~$u$ of the homogeneous equation~\eqref{LDE} on~$\mathbb Z^+_{n-r}$. Thus, $x$ also satisfies the homogeneous equation~\eqref{LDE} for $k\geq n$ and hence $z(k)=0$ for $k\geq n$. It remains to consider the case when $m\leq k\leq n-1$. Let such a~$k$ be fixed. By definition, we have
\[
\begin{split}
    z(k)&=x(k+1)-L_k(x_k)\\
    &=\chi(k+1)u(k+1)-L_k(\chi_k u_k)\\
    &=\chi(k+1)L_k(u_k)-L_k(\chi_k u_k)\\
    &=L_k(\chi(k+1)u_k-\chi_ku_k).
\end{split}
\]
Therefore, using \eqref{funibdd}, it follows that
\begin{equation*}\label{eq: est zk}
    |z(k)|\leq M\|\chi(k+1)u_k-\chi_ku_k\|.
\end{equation*}
Let $\theta\in[-r,0]\cap\mathbb Z$. Then \[
\begin{split}
    |\left(\chi(k+1)u_k-\chi_ku_k\right)(\theta)|&=|\left(\chi(k+1)-\chi (k+\theta)\right) u(k+\theta)|,
\end{split}
\]
and hence
\[
\begin{split}
    |\left(\chi(k+1)u_k-\chi_ku_k\right)(\theta)|&\leq \sum_{j=m}^{k} \|u_j\|^{-1}|u(k+\theta)| \qquad \text{whenever } k+\theta\leq m
\end{split}
\]
and 
\[
\begin{split}
    |\left(\chi(k+1)u_k-\chi_ku_k\right)(\theta)|&\leq \sum_{j=k+\theta}^{k} \|u_j\|^{-1}|u(k+\theta)| \qquad \text{whenever } k+\theta> m.
\end{split}
\]
In view of~\eqref{eq: invariance S U}, $u_m=\phi\in\mathcal S(m)$ implies that $u_j=T(j,m)\phi\in\mathcal S(j)$ for $j\geq m$. Therefore, by Claim~\ref{C}, for $m\leq j\leq k$, we have
\[
|u(k+\theta)|\leq\|u_{k}\|=\|T(k,j)u_j\|\leq Q\|u_j\|
\]
so that $\|u_j\|^{-1}|u(k+\theta)|\leq Q$. From this, we conclude that if $k+\theta\leq m$ so that $k-m\leq-\theta\leq r$, then
\[
\sum_{j=m}^{k} \|u_j\|^{-1}|u(k+\theta)|
\leq Q(k-m+1)\leq Q(r+1),
\]
while in case $k+\theta>m$, we have
\[
\sum_{j=k+\theta}^{k} \|u_j\|^{-1}|u(k+\theta)|
\leq Q(-\theta+1)\leq Q(r+1).
\]
Therefore, in both cases $k+\theta\leq m$ and $k+\theta>m$, we have that
\[
\begin{split}
    |\left(\chi(k+1)u_k-\chi_ku_k\right)(\theta)|&\leq Q(r+1).
\end{split}
\]
Since $\theta\in[-r,0]\cap\mathbb Z$ was arbitrary, this implies that 
\[
\|\chi(k+1)u_k-\chi_k u_k\|\leq Q(r+1),
\]
which, combined with~\eqref{eq: est zk}, yields
\[
|z(k)|\leq MQ(r+1).
\]
We have shown that the last inequality is valid whenever $m\leq k\leq n-1$ and $z(k)=0$ otherwise. Hence 
\[
\sup_{k\geq0}|z(k)|\leq MQ(r+1).
\]
This, together with
\eqref{eq: est x z}, implies that
\begin{equation}\label{xineq}
\sup_{k\geq-r}|x(k)|\leq CMQ(r+1).
\end{equation}
Since $n-r\geq n-N\geq m$ and~$\chi$ is nondecreasing on~$\mathbb Z^+_{-r}$, we have for $\theta\in[-r,0]\cap\mathbb Z$,
\[
|x(n+\theta)|=\chi(n+\theta)|u(n+\theta)|\geq\chi(n-r)|u(n+\theta)|
=|u(n+\theta)|\sum_{j=m}^{n-r-1} \|u_j\|^{-1}.
\]
Hence
\[
\|x_n\|\geq\|u_n\|\sum_{j=m}^{n-r-1} \|u_j\|^{-1}.
\]
According to Claim~\ref{C}, $u_m=\phi\in\mathcal S(m)$ implies that  
 $\|u_j\|=\|T(j,m)\phi\|\leq Q\|\phi\|$ for $j\geq m$. This, together with the previous inequality, yields
 \[
 \|x_n\|\geq\|u_n\|\sum_{j=m}^{n-r-1} \|u_j\|^{-1}
 \geq\|u_n\|\frac{n-m-r}{Q\|\phi\|}.
 \]
 The last inequality, combined with~\eqref{lowup} and~\eqref{xineq}, implies that
 \[
\begin{split}
    CMQ(r+1)\geq \|x_n\| 
    \geq \|u_n\|\frac{n-m-r}{Q\|\varphi\|}
    > \frac{n-m-r}{eQ}.
\end{split}
\]
However, this contradicts the fact that $n\geq m+N$ with $N$ satisfying \eqref{Nchoice}. Thus, \eqref{e} holds whenever $m\in \mathbb Z^+_0$ and $\phi\in\mathcal S(m)$.

Using~\eqref{e}, we can easily complete the proof. Choose an integer~$N$  satisfying~\eqref{Nchoice}. Let $n\geq m$. Then, $n-m=kN+h$ for some $k\in \mathbb Z^+_0$ and $0\leq h\leq N-1$. From~\eqref{e} and Claim~\ref{C}, we obtain for $\phi \in \mathcal S(m)$, 
\[
\begin{split}
\|T(n,m)\phi\| &=\|T(m+kN+h, m)\phi \|\\
&=\|T(m+kN+h,m+kN)T(m+kN,m)\phi\|\\
&\le Q\|T(m+kN, m)\phi \|\\
&\le Qe^{-k}\|\phi \|\\
&\le eQe^{-\frac{n-m}{N}}\|\phi \|.
\end{split}
\]
Hence, the conclusion of the claim holds with $D=eQ$ and $\lambda=1/N$. 
\end{proof}

\begin{Claim}\label{C-U}
There exists $Q'>0$ such that 
\[
\|T(n,m)\phi \| \le Q'\|\phi \|, 
\]
for every $n,m\in \mathbb Z^+_0$ with $n\leq m$ and $\phi \in \mathcal U(m)$.
\end{Claim}
\begin{proof}[Proof of Claim~\ref{C-U}]
Given $m\in\mathbb Z^+_0$ and $\phi \in \mathcal U(m)$, there exists $\bar \phi \in \mathcal{U}$ such that $\phi=T(m,0)\bar \phi$.  Let $u\colon\mathbb Z^+_{-r}\to X$ be the solution of Eq.~\eqref{LDE} such that $u_0=\bar \phi$. Consider $x\colon \mathbb Z^+_{-r}\to X$ and $z\colon \mathbb Z^+_0\to X$ given by 
\[
x(k)=\begin{cases}
u(k)& \text{for $-r\leq k\leq m$}\\[10pt]
0 & \text{for $k\geq m+1$,}
\end{cases}
\]
and 
\[
z(k)=x(k+1)-L_k(x_k), \qquad  k\in \mathbb Z^+_0,
\]
so that \eqref{adm} is satisfied. Moreover, since $x(k)=0$ for $k\geq{m+1}$, it follows that $\sup_{k\geq {-r}}|x(k)|<\infty$.
Furthermore, $x_0=u_0\in\mathcal U$ and $z(k)=0$ for $0\leq k\leq m-1$ 
 and $k\geq m+r+1$. Proceeding as in the proof of Claim~\ref{C}, it can be shown that 
\[
\sup_{k\geq 0}|z(k)|\leq  2Me^{\omega r}\| \phi\|.
\]
From conclusion~\eqref{cgt} of Claim~\ref{LEM2}, we conclude that 
\[
\sup_{-r\leq k\leq m}|u(k)| \le \sup_{k\geq{-r}} |x(k)|  \le C \sup_{k\geq 0}|z(k)|\leq  2CMe^{\omega r}\| \phi\|.
\]
This implies that the conclusion of the claim holds with 
\[
Q':= 2CMe^{\omega r}>0.
\]
\end{proof}

\begin{Claim}\label{exp est unstable}
There exist $D', \lambda' >0$ such that 
\[
\|T(n,m)\phi \| \le D'e^{-\lambda' (m-n)}\|\phi\|
\]
for every $n,m\in \mathbb Z^+_0$ with $n\leq m$ and $\phi \in \mathcal U(m)$.
\end{Claim}

\begin{proof}[Proof of Claim~\ref{exp est unstable}]
We claim that if
\begin{equation}\label{NchoiceU}
N'>eCM(Q')^2(r+1)
\end{equation}
with~$Q'$ as in Claim~\ref{C-U}, then for every $m\geq N'$ and $\phi\in\mathcal U(m)$, 
\begin{equation}\label{eU}
\|T(n,m)\phi \| \le \frac 1 e \|\phi \|\qquad\text{whenever $0\leq n\leq m- N'$}.
\end{equation}
Suppose, for the sake of contradiction, that \eqref{NchoiceU} holds and there exist $m\geq N'$ and $\phi\in\mathcal U(m)$ such that
\begin{equation}\label{lowU}
\|T(n,m)\phi \|>\frac{1}{ e }\|\phi \|\qquad\text{for some~$n$ with $0\leq n\leq m-N'$}.
\end{equation}
Fix~$n$ with $0\leq n\leq m-N'$ such that the first inequality in~\eqref{lowU} holds.                        Evidently, \eqref{lowU} implies that $\phi\in\mathcal U(m)$ is nonzero. Therefore, there exists a nonzero $\bar \phi\in \mathcal{U}$ such that $\phi =T(m,0)\bar \phi$.
Let~$u$ denote the unique solution of Eq.~\eqref{LDE} with $u_0=\bar \phi$ so that $u_m=T(m,0)u_0=T(m,0)\bar\phi=\phi$. 
Since $u_n=T(n,m)u_m=T(n,m)\phi$, the first inequality in~\eqref{lowU} can be written as
\begin{equation}\label{lowupU}
\|u_{n}\|>\frac 1 e \|\phi\|.
\end{equation}
Choose a sequence $\psi\colon\mathbb Z^+_0\to[0,1]$ such that 
\begin{equation}\label{psidef}
\psi(j)=1\quad\text{for $0\leq j\leq m$}\qquad\text{and}\qquad\psi(j)=0\quad\text{for $j\geq m+1$}.
\end{equation}
By Claim~\ref{lem: invertibility}, $0\neq\bar\phi\in\mathcal U$ implies that $u_j=T(j,0)\bar\phi\neq0$ for $j\geq0$. Therefore, we can define a function $x\colon \mathbb Z^+_{-r}\to X$  by
\begin{equation*}
x(k)=\chi(k)u(k) \qquad \text{for $k\geq{-r}$},
\end{equation*}
where
\[
 \chi(k)=\begin{cases}
 \displaystyle \sum_{j=0}^{\infty} \psi(j)\|u_j\|^{-1}\qquad&\text{for $-r\leq k\leq0$}, \\[15pt]
 \displaystyle\sum_{j=k}^{\infty} \psi(j)\|u_j\|^{-1}
 \qquad&\text{for $k\geq1$}.
 \end{cases}
 \]
Note that 
\[x_0=cu_0=c\bar\phi\in\mathcal U,\qquad\text{where $c=\chi(0)=\sum_{j=0}^{m}\psi(j) \|u_j\|^{-1}$.}
\] 
Since $\psi(k)=0$ and hence $x(k)=0$ for $k\geq m+1$, we have that $\sup_{k\geq {-r}}|x(k)|<\infty $. Moreover, $x$ satisfies~\eqref{adm} with $z\colon \mathbb Z^+_0\to X$ defined by
\begin{equation}\label{eq: z aux}
    z(k)=x(k+1)-L_k(x_k),\qquad k\geq 0.   
\end{equation}
Since $x_k=0$ for $k\geq m+r+1$, it follows that $z(k)=0$ for $k\geq{m+r+1}$. In particular, $z$ is bounded on~$\mathbb Z^+_0$.  From~\eqref{eq: z aux} and Claim~\ref{LEM2}, we conclude that
\begin{equation}\label{eq: est x z U}
\sup_{k\geq{-r}}|x(k)|\leq C\sup_{k\geq{0}}|z(k)|.
\end{equation}
Let $k\geq0$ be arbitrary. By the same calculations as in the proof of Claim~\ref{C}, we have 
\[
\begin{split}
    z(k)&=L_k(\chi(k+1)u_k-\chi_ku_k).
\end{split}
\]
From this and~\eqref{funibdd}, we find that
\begin{equation}\label{eq: est zk U}
    |z(k)|\leq M\|\chi(k+1)u_k-\chi_ku_k\|.
\end{equation}
Let $\theta\in[-r,0]\cap\mathbb Z$. Then
\[
\begin{split}
    |\left(\chi(k+1)u_k-\chi_ku_k\right)(\theta)|&=|\left(\chi(k+1)-\chi (k+\theta)\right) u(k+\theta)|.
\end{split}
\]
From this and the definition of~$\chi$, taking into account that $0\leq\psi\leq1$ on $\mathbb Z^+_0$,  we conclude that
\begin{equation}\label{firstcase}
  |\left(\chi(k+1)u_k-\chi_ku_k\right)(\theta)|\leq \sum_{j=0}^{k} \|u_j\|^{-1}|u(k+\theta)| \qquad \text{whenever $k+\theta\leq 0$}.
\end{equation}
and
\begin{equation}\label{secondcase}
    |\left(\chi(k+1)u_k-\chi_ku_k\right)(\theta)|\leq \sum_{j=k+\theta}^{k} \|u_j\|^{-1}|u(k+\theta)| \qquad \text{whenever $k+\theta>0$}.
\end{equation}
By Claim~\ref{lem: invertibility}, $u_0=\bar\phi\in\mathcal U$ implies that $u_j=T(j,0)u_0\in\mathcal U(j)$ for $j\geq 0$. Therefore,  according to Claim~\ref{C-U}, if $k+\theta<0$, then 
\[
|u(k+\theta)|\leq\|u_{0}\|=\|T(0,j)u_j\|\leq Q'\|u_j\|
\]
and hence $\|u_j\|^{-1}|u(k+\theta)|\leq Q'$ for $j\geq0$. If $k+\theta\leq 0$ so that $k\leq-\theta\leq r$, then from the last ineaquality and~\eqref{firstcase}, we obtain that
\begin{equation}\label{fest}
  |\left(\chi(k+1)u_k-\chi_ku_k\right)(\theta)|\leq Q'(r+1) \qquad \text{whenever $k+\theta\leq0$}.
\end{equation}
It follows by similar arguments that if $k+\theta>0$, then 
\[
|u(k+\theta)|\leq\|u_{k+\theta}\|=\|T(k+\theta,j)u_j\|\leq Q'\|u_j\|
\]
and hence  $\|u_j\|^{-1}|u(k+\theta)|\leq Q'$ for $j\geq k+\theta$. This, together with~\eqref{secondcase}, yields
\begin{equation}\label{sest}
  |\left(\chi(k+1)u_k-\chi_ku_k\right)(\theta)|\leq Q'(r+1) \qquad \text{whenever $k+\theta>0$}.
\end{equation}
 Since $\theta\in [-r,0]\cap\mathbb Z$ was arbitrary, \eqref{fest} and~\eqref{sest} imply that
\[
\begin{split}
    \|\chi(k+1)u_k-\chi_ku_k\|\leq Q'(r+1).
\end{split}
\]
Since $k\geq 0$ was arbitrary, the last inequality, combined with~\eqref{eq: est zk U}, implies that
\[
\sup_{k\geq0}|z(k)|\leq MQ'(r+1).
\]
This, together with~\eqref{eq: est x z U}, yields 
\begin{equation}\label{impest}
\sup_{k\geq-r}|x(k)|\leq CMQ'(r+1).
\end{equation}
Since $\chi$ is nonincreasing on $\mathbb Z^+_{-r}$, we have for $\theta\in[-r,0]\cap\mathbb Z$,
\[
|x(n+\theta)|=\chi(n+\theta)|u(n+\theta)|\geq\chi(n)|u(n+\theta)|=|u(n+\theta)|\sum_{j=n}^m \|u_j\|^{-1},
\]
the last equality being a consequence of~\eqref{psidef}. Hence
\[
\|x_n\|\geq\|u_n\|\sum_{j=n}^m \|u_j\|^{-1}.
\]
Since $u_m=\phi\in\mathcal U(m)$, by Claim~\ref{C-U}, we have that $\|u_j\|=\|T(j,m)\phi\|\leq Q'\|\phi\|$ for $0\leq j\leq m$.
This, together with the previous inequality, gives
\[
\|x_n\|\geq\|u_n\|\frac{m-n+1}{Q'\|\phi\|}.
\]
This, combined with~\eqref{lowupU} and~\eqref{impest}, yields
\[
\begin{split}
   CMQ'(r+1)\geq \|x_n\|\geq\|u_n\|\frac{m-n+1}{Q'\|\phi\|}>
   \frac{m-n+1}{eQ'}.
\end{split}
\]
The last inequality contradicts the fact that $n-m\geq N'$ with $N'$ satisfying \eqref{NchoiceU}. Thus, \eqref{eU} holds whenever $m\geq N'$ and $\phi\in\mathcal U(m)$.

Now, using~\eqref{eU}, we can easily complete the proof.  Let  $0\leq n\leq m$ and $\phi \in \mathcal U(m)$. Choose an integer~$N'$ satisfying~\eqref{NchoiceU}. Then, $m-n=kN'+h$ for some $k\in \mathbb Z^+_0$ and $0\leq h\leq N'-1$. From~\eqref{eU} and Claim~\ref{C-U}, we obtain  
\[
\begin{split}
\|T(n,m)\phi\| &=\|T(n, n+kN'+h)\phi \|\\
&=\|T(n, n+kN')T(n+kN',n+kN'+h)\phi \| \\
&\le e^{-k}\|T(n+kN',n+kN'+h)\phi \|\\
&\le Q'e^{-k}\|\phi \|\\
&\le eQ'e^{-\frac{m-n}{N'}}\|\phi \|.
\end{split}
\]
Thus, the conclusion of the claim holds with $D'=eQ'$ and $\lambda'=1/N'$.
\end{proof}

For each $n\in \mathbb Z^+_0$, let $P_n$ denote  the projection of~$\mathcal B_r$ onto $\mathcal S(n)$ along~$\mathcal U(n)$ associated with the decomposition \eqref{SPLImainresult1}.

\begin{Claim}\label{Claim: projectionsbounded}
The projections $P_n$, $n\in \mathbb Z^+_0$, are uniformly bounded, i.e.,
\begin{equation*}
\sup_{n\geq 0}\|P_n\|<\infty.
\end{equation*}
\end{Claim}

\begin{proof}[Proof of Claim~\ref{Claim: projectionsbounded}] Since $\operatorname{ker}P_n=\mathcal U(n)$ and $\operatorname{im}P_n=\mathcal S(n)$ for $n\in\mathbb Z^+_0$, Claims~\ref{exp est stable} and~\ref{exp est unstable} show that the hypotheses of Lemma~3.1 of Huy and Van Minh~\cite{HM} are satisfied with $X=\mathcal B_r$ and $A_n=T(n+1,n)$. Therefore, the desired conclusion follows from \cite[Lemma~3.1]{HM}. 
 \end{proof}

Now we can complete the proof of Theorem \ref{mainresult1}.
Let $\phi \in \mathcal B_r$ and $n,m\in \mathbb Z^+_0$ with $n\geq m$ be fixed. From $P_m\phi\in\mathcal S(m)$ and~\eqref{eq: invariance S U}, we have that $T(n,m)P_m\phi\in\mathcal S(n)$. Hence
\[
P_n T(n,m)P_m\phi=T(n,m)P_m\phi.
\]
Similarly, considering $Q_m=I-P_m$, $Q_m\phi\in\mathcal U(m)$ and~\eqref{eq: invariance S U} imply that $T(n,m)Q_m\phi\in\mathcal U(n)$. Hence
\[
P_n T(n,m)Q_m\phi=0.
\]
From the above relations, taking into account that $\phi=P_m\phi+Q_m\phi$, we conclude that
\[
P_nT(n,m)\phi=P_nT(n,m)P_m\phi+P_nT(n,m)Q_m\phi=T(n,m)P_m\phi.
\]
Since $\phi\in\mathcal B_r$ was arbitrary, this proves~\eqref{pro}.

Evidently, $\operatorname{ker}P_m=\mathcal U(m)$ for $m\in\mathbb Z^+$. Therefore, from
Claim~\ref{lem: invertibility} and~\eqref{eq: invariance S U}, it follows that the restriction $T(n,m)\rvert_{\operatorname{ker} P(m)} \colon \operatorname{ker} P(m)\to \operatorname{ker} P(n)$ is invertible and onto. Furthermore, by Claim \ref{Claim: projectionsbounded}, the projections $P_n$, $n\in \mathbb Z^+_0$, are uniformly bounded. Combining this fact with Claims \ref{exp est stable} and \ref{exp est unstable}, we conclude that the exponential estimates~\eqref{eq: def est stable} and \eqref{eq: def est unstable} are also satisfied. Thus, \eqref{LDE} admits an exponential dichotomy.  
\end{proof}

\begin{remark}\label{r4}
In the proof of Theorem~\ref{mainresult1}, we have shown that the Perron property (see Claim~\ref{LEM1}) implies the existence of an exponential dichotomy for  Eq.~\eqref{nef}. Results of this type have a long history that goes back to the pioneering works of Perron~\cite{Perron} for ordinary differential equations and Li~\cite{Li} for difference equations. Subsequent important contributions are due to Massera and Sch\"affer~\cite{MS1, MS2}, Dalecki\u \i  \ and   Kre\u \i n~\cite{DK}, Coppel~\cite{Coppel} and Henry~\cite{Henry}, who was the first to consider the case of noninvertible dynamics. For more recent contributions, we refer to~\cite{Huy, HM, LRS, MRS, SS1, SS2, SS3} and the references therein. A comprehensive overview of the relationship between hyperbolicity and the Perron property is given in~\cite{BDV}.
\end{remark}

\section{Shadowing of linear Volterra difference equations with infinite delay}\label{veid}
In this section, we are interested in the shadowing of the Volterra difference equation with infinite delay~\eqref{vde}, where the kernel~$A$ satisfies condition~\eqref{gammacond}.
The phase space for Eq.~\eqref{vde} is the Banach space $(\mathcal B_\gamma,\|\cdot\|)$ given by
\begin{equation*}
	\mathcal B_\gamma=\biggl\{\,\phi\colon\mathbb Z^-_0\to\mathbb C^d:\sup_{\theta\in\mathbb Z^-_0}|\phi(\theta)|e^{\gamma\theta}<\infty\,\biggr\},
	\qquad \|\phi\|=\sup_{\theta\in\mathbb Z^-_0}|\phi(\theta)| e^{\gamma\theta},\qquad\phi\in\mathcal B_\gamma.
\end{equation*}
Under condition~\eqref{gammacond}, Eq.~\eqref{vde} can be written equivalently 
in the form
\begin{equation}\label{mvde}
	x(n+1)=L(x_n),
\end{equation}
where  $x_n\in\mathcal B_\gamma$ is defined by $x_n(\theta)=x(n+\theta)$ for $\theta\in\mathbb Z^-_0$ and
$L\colon\mathcal B_\gamma\to\mathbb C^d$ is a bounded linear functional defined by
	\begin{equation*}
	L(\phi)=\sum_{j=0}^\infty A(j)\phi(-j),\qquad\phi\in\mathcal B_\gamma.
\end{equation*}
It is known (see, e.g.,~\cite{MMNN}, \cite{N}) that if~\eqref{gammacond} holds, then for every $\phi\in\mathcal B_\gamma$, there exists a unique function $x\colon\mathbb Z\to\mathbb C^d$ satisfying Eq.~\eqref{vde} (equivalently, Eq.~\eqref{mvde}) such that $x(\theta)=\phi(\theta)$ for $\theta\in\mathbb Z^-_0$. We shall call~$x$ the \emph{solution of Eq.~\eqref{vde} (or~\eqref{mvde}) on~$\mathbb Z^+_0$ with initial value $x_0=\phi$}. By a \emph{solution of Eq.~\eqref{vde} on~$\mathbb Z^+_0$}, we mean a solution~$x$ with initial value~$x_0=\phi$ for some $\phi\in\mathcal B_\gamma$.

For Eq.~\eqref{vde}, the definition of shadowing can be modified as follows.

\begin{definition}\label{voltshad}
We say that Eq.~\eqref{vde} is \emph{shadowable on~$\mathbb Z^+_0$} if for each $\epsilon >0$ there exists $\delta>0$ such that for every function $y\colon \mathbb Z\to \mathbb C^d$ satisfying
\[
\sup_{n\geq 0}|y(n+1)-L(y_n)| \le \delta, 
\]
there exists a solution $x$ of~\eqref{vde} on~$\mathbb Z^+_0$ such that  
\[
\sup_{n\geq 0} \|x_n-y_n\| \le \epsilon.
\]
\end{definition}
	
The main result of this section is the following theorem which shows that, under condition~\eqref{gammacond}, Eq.~\eqref{vde} is shadowable on~$\mathbb Z^+_0$ if and only if it is hyperbolic.

\begin{theorem}\label{mainresult2}
Suppose that~\eqref{gammacond} holds. Then, the following statements are equivalent.
\vskip5pt
\emph{(i)} Eq.~\eqref{vde} is shadowable on~$\mathbb Z^+_0$;
\vskip3pt

\vskip3pt

\emph{(ii)} The characteristic equation~\eqref{chareq} has no root on the unit circle $|\lambda|=1$.

\end{theorem}

Before we give a proof of Theorem~\ref{mainresult2},  we summarize some facts from the spectral theory of linear Volterra difference equations with infinite delay (\cite{MMNN}, \cite{M}, \cite{N}).  

For each $n\in\mathbb Z^+_0$, define 
$T(n)\colon\mathcal B_\gamma\to\mathcal B_\gamma$ by $T(n)\phi=x_n(\phi)$ for $\phi\in\mathcal B_\gamma$, where $x(\phi)$ is the unique solution of Eq.~\eqref{vde} on~$\mathbb Z^+_0$ with initial value $x_0(\phi)=\phi$. It is well-known (see~\cite{MMNN, M}) that $T(n)$ is a bounded linear operator in~$\mathcal B_\gamma$ which has the semigroup property $T(0)=I$, the identity on~$\mathcal B_\gamma$, and $T(n+m)=T(n)T(m)$ for $n,m\in\mathbb Z^+_0$. As a consequence, we have that
\begin{equation*}
	T(n)=T^n,\qquad n\in\mathbb Z^+_0,\qquad\text{where\quad $T:=T(1)$}.
\end{equation*}
From the definition of the \emph{solution operator} $T=T(1)$ and Eq.~\eqref{vde}, we have that
\begin{equation}\label{tdef}
	[T(\phi)](\theta)=\begin{cases}
	\displaystyle\sum_{j=0}^\infty A(j)\phi(-j)&\qquad\text{for $\theta=0$,}\\
	\phi(\theta+1)&\qquad\text{for $\theta\leq-1$}.
	\end{cases}	
\end{equation}
If~\eqref{gammacond} holds, then the characteristic function~$\varDelta$ defined by~\eqref{charfunc} is an analytic function of the complex variable~$\lambda$ in the region $|\lambda|>e^{-\gamma}$.
Denote by~$\varSigma$ the set of characteristic roots of Eq.~\eqref{vde},
\begin{equation*}
	\varSigma=\{\,\lambda\in\mathbb C:|\lambda|>e^{-\gamma},\,\det\varDelta(\lambda)=0\,\},
\end{equation*}
and define
\begin{equation*}
	\varSigma^{cu}=\{\,\lambda\in\varSigma:|\lambda|\geq1\}.
\end{equation*}
It follows from the analycity of~$\varDelta$ that $\varSigma^{cu}$ is a finite spectral set for~$T$. The corresponding spectral projection~$\varPi^{cu}$ on~$\mathcal B_\gamma$ is defined by
\begin{equation*}
	\varPi^{cu}=\frac{1}{2\pi i}\int_C (\lambda I-T)^{-1}\,d\lambda,
\end{equation*}
where~$C$ is any rectifiable Jordan curve which is disjoint with~$\varSigma$ and contains~$\varSigma^{cu}$ in its interior, but no point of~$\varSigma^s:=\varSigma\setminus\varSigma^{cu}$. The phase space~$\mathcal B_\gamma$ can be decomposed into the direct sum
\begin{equation*}
	\mathcal B_\gamma=\mathcal B_\gamma^{cu}\oplus\mathcal B_\gamma^s
\end{equation*}
with $\mathcal B_\gamma^{cu}=\varPi^{cu}(\mathcal B_\gamma)$ and $\mathcal B_\gamma^{s}=\varPi^{s}(\mathcal B_\gamma)$, where $\varPi^{s}=I-\varPi^{cu}$. The subspaces $\mathcal B_\gamma^{cu}$ and $\mathcal B_\gamma^{s}$ are called the \emph{center-unstable subspace} and the \emph{stable subspace} of~$\mathcal B_\gamma$, respectively.
The spectra of the restrictions $T^{cu}:=T|_{\mathcal B_{\gamma}^{cu}}$ and $T^s:=T|_{\mathcal B_{\gamma}^{s}}$ satisfy 
\begin{equation*}
	\sigma(T^{cu})=\varSigma^{cu}\quad \text{ and } \quad\sigma(T^s)=\sigma(T)\setminus\varSigma^{cu}=\{\,\lambda\in\sigma(T):|\lambda|<1\,\}.
\end{equation*}
If $\varSigma^{cu}$ is nonempty, then it consists of finitely many eigenvalues of~$T$,
\begin{equation*}
	\varSigma^{cu}=\{\,\lambda_1,\dots,\lambda_r\,\},
\end{equation*}
and $\mathcal B_\gamma^{cu}$ can be written as a direct sum of the nullspaces
\begin{equation*}
	\mathcal B_\gamma^{cu}=\operatorname{ker}((T-\lambda_1 I)^{p_1})\oplus\cdots\oplus \operatorname{ker}((T-\lambda_r I)^{p_r}),
\end{equation*}
where $p_j$ is the index (ascent) of~$\lambda_j$, $j=1,\dots,r$. (see \cite[Remark~2.1, p.~62]{MMNN}).

An explicit representation of the spectral projection~$\varPi^{cu}$ can be given using the duality between Eq.~\eqref{vde} and its \emph{formal adjoint equation}
\begin{equation}\label{fae}
	y(n-1)=\sum_{j=0}^\infty y(n+j) A(j),\qquad n\in\mathbb Z^+_0,
\end{equation}
where $y(n)\in\mathbb C^{d*}$. Here $\mathbb C^{d*}$ denotes the $d$-dimensional space of complex row vectors with a norm~$|\cdot|$ which is compatible with the given norm on~$\mathbb C^d$, i.e., $|x^*x|\leq|x^*||x|$ for all $x\in\mathbb C^d$. The superscript ${}^*$ indicates the conjugate transpose. The phase space for the formal adjoint equation~\eqref{fae} is the Banach space 
$(\mathcal B_{\tilde\gamma}^\sharp,\|\cdot\|)$ defined by
\begin{equation*}
	\mathcal B_{\tilde\gamma}^\sharp=\biggl\{\,\psi\colon\mathbb Z^+_0\to\mathbb C^{d*}:\sup_{\zeta\in\mathbb Z^+_0}|\psi(\zeta)|e^{-\tilde\gamma\zeta}<\infty\,\biggr\},
	\quad \|\psi\|=\sup_{\zeta\in\mathbb Z^+_0}|\psi(\zeta)| e^{-\tilde\gamma\zeta},\quad\psi\in\mathcal B_{\tilde\gamma}^\sharp,
\end{equation*}
where $\tilde\gamma$ is a fixed number such that $0<\tilde\gamma<\gamma$. The solution operator $T^\sharp\colon\mathcal B_{\tilde\gamma}^\sharp\to\mathcal B_{\tilde\gamma}^\sharp$ of Eq.~\eqref{fae} is given by
\begin{equation}\label{tsharpdef}
	[T^\sharp(\psi)](\zeta)=\begin{cases}
	\displaystyle\sum_{j=0}^\infty \psi(j)A(j)&\qquad\text{if $\zeta=0$,}\\
	\psi(\zeta-1)&\qquad\text{if $\zeta\geq1$}.
	\end{cases}	
\end{equation}
Define a \emph{bilinear form} $\langle\cdot,\cdot\rangle\colon\mathcal B_{\tilde\gamma}^\sharp\times \mathcal B_\gamma\to\mathbb C$ by
\begin{equation*}
\langle\psi,\phi\rangle=\psi(0)\phi(0)+\sum_{j=1}^\infty\sum_{\zeta=0}^{j-1}\psi(\zeta+1)A(j)\phi(\zeta-j),
\qquad\phi\in\mathcal B_\gamma,\quad\psi\in	\mathcal B_{\tilde\gamma}^\sharp.
\end{equation*}
As shown in \cite[Ineq.~(3.3), p.~64]{MMNN}, this bilinear form is bounded, i.e., there exists $K>0$ such that
\begin{equation}\label{boudbil}
	|\langle\psi,\phi\rangle|\leq K\|\psi\|\|\phi\|,
	\qquad\phi\in\mathcal B_\gamma,\quad\psi\in	\mathcal B_{\tilde\gamma}^\sharp.
\end{equation}
Moreover, between Eqs.~\eqref{vde} and~\eqref{fae}, we have the following \emph{duality relation} (see \cite[Lemma~3.1]{MMNN})
\begin{equation}\label{dual}
\langle\psi,T\phi\rangle=\langle T^\sharp\psi,\phi\rangle,\qquad \phi\in\mathcal B_\gamma,\quad\psi\in\mathcal B_{\tilde\gamma}^\sharp.
\end{equation}
It is known that~$T$ and~$T^\sharp$ have the same spectrum and the dimension of the subspace
\begin{equation*}
	\mathcal N^\sharp:=\operatorname{ker}((T^\sharp-\lambda_1 I)^{p_1})\oplus\cdots\oplus \operatorname{ker}((T^\sharp-\lambda_r I)^{p_r})
\end{equation*}
of~$\mathcal B_{\tilde\gamma}^\sharp$ is the same as the (finite) dimension of the center-unstable subspace~$\mathcal B_\gamma^{cu}$, which will be denoted by~$s$. Let $\{\,\phi_1,\dots,\phi_s\,\}$ and $\{\,\psi_1,\dots,\psi_s\,\}$ be bases for~$\mathcal B_\gamma^{cu}$ and~$\mathcal N^\sharp$, respectively. Define $\varPhi=(\phi_1,\dots,\phi_s)$ and $\varPsi=\operatorname{col}(\psi_1,\dots,\psi_s)$.
Then, the $s\times s$ matrix $\langle\varPsi,\varPhi\rangle$ given by $\langle\varPsi,\varPhi\rangle=(\langle\psi_i,\phi_j\rangle)_{i,j=1,\dots,s}$
is nonsingular, therefore, by replacing~$\varPsi$ with $\langle\varPsi,\varPhi\rangle^{-1}\varPsi$, we may (and do) assume that $\langle\varPsi,\varPhi\rangle=E$, the $s\times s$ unit matrix. The projection
$\varPi^{cu}\colon\mathcal B_\gamma\to\mathcal B_\gamma^{cu}$ can be given explicitly by (see \cite[Theorem~3.1]{MMNN})
\begin{equation}\label{exproj}
	\varPi^{cu}\phi=\varPhi\langle\varPsi,\phi\rangle,\qquad\phi\in\mathcal B_\gamma,
\end{equation}
where $\langle\varPsi,\phi\rangle$ denotes the column vector $\operatorname{col}(\langle\psi_1,\phi\rangle,\dots,\langle\psi_s,\phi\rangle)$.

The subspace~$\mathcal B_\gamma^{cu}$ is invariant under the solution operator~$T$. If~$B$ 
denotes the representation matrix of the linear transformation $T|_{\mathcal B_\gamma^{cu}}$ with respect to the basis~$\varPhi$ of~$\mathcal B_\gamma^{cu}$, then 
\begin{equation}\label{tphi}	
T\varPhi=\varPhi B\qquad\text{and}\qquad \sigma(B)=\varSigma^{cu}.
\end{equation}
 A similar argument yields the existence of a square matrix~$C$
such that 
\begin{equation}\label{tsharpsi}	
T^\sharp\varPsi=C\varPsi\qquad\text{and}\qquad \sigma(C)=\varSigma^{cu}.
\end{equation}
Now suppose that $x$ is a solution of the nonhomogeneous equation
\begin{equation}\label{nvde}
x(n+1)=L(x_n)+p(n), \qquad n\in \mathbb Z^+_0.
\end{equation}
with initial value $x_0=\phi$ for some $\phi\in\mathcal B_\gamma$. Then, $x$ satisfies the following representation formula in~$\mathcal B_\gamma$ (see \cite[Theorem~2.1]{M}), which is called the \emph{variation of constants formula} for Eq.~\eqref{nvde} in the phase space, 
\begin{equation}\label{vcf}
x_n=T(n)\phi+\sum_{j=0}^{n-1}T(n-1-j)\varGamma p(j),\qquad n\in\mathbb Z^+_0,	
\end{equation}
where the operator $\varGamma\colon\mathbb C^d\to\mathcal B_\gamma$ is defined by
\begin{equation}\label{Gdef}
	[\varGamma x](\theta)=\begin{cases}
	x&\qquad\text{if $\theta=0$,}\\
	0&\qquad\text{if $\theta\leq-1$}.
	\end{cases}	
\end{equation}
Evidently, 
\begin{equation}\label{Gamisom}
	\|\varGamma x\|=|x|,\qquad x\in\mathbb C^d.
\end{equation}
Finally, let~$z(n)$ be the coordinate of the projection $\varPi^{cu}x_n$ with respect to the basis~$\varPhi$, i.e., $\varPi^{cu}x_n=\varPhi z(n)$ for $n\in\mathbb Z^+_0$. In view of~\eqref{exproj}, $z(n)$ is given explicitly by\begin{equation}\label{zdef}
z(n)=\langle\varPsi,x_n\rangle,\qquad n\in\mathbb Z^+_0.
	\end{equation}
Moreover, it is known (see \cite[Theorem~3]{N}) that $z$ satisfies the following first order difference equation in~$\mathbb C^s$,
\begin{equation}\label{cueq}
	z(n+1)=Bz(n)+\langle\varPsi,\varGamma p(n)\rangle,\qquad n\in\mathbb Z^+_0,
	\end{equation}
with~$B$ as in~\eqref{tphi}.

Now we are in a position to give a proof of Theorem~\ref{mainresult1}.
\begin{proof}[Proof of Theorem~\ref{mainresult1}]

$(i)\Rightarrow(ii)$. Suppose, for the sake of contradiction, that Eq.~\eqref{vde} is shadowable, but~(ii) does not hold. The shadowing property of Eq.~\eqref{vde} implies the following  Perron type property.
\begin{Claim}\label{fc}
For every bounded function $p\colon \mathbb Z^+_0\to \mathbb C^d$, 
there exists a function $x\colon \mathbb Z\to\mathbb C^d$ satisfying the nonhomogeneous equation~\eqref{nvde} with
\begin{equation}\label{bdd}
	\sup_{n\geq 0}|x(n)|<\infty.
\end{equation}
\end{Claim}
The  proof of Claim~\ref{fc} is almost identical with that of Claim~\ref{LEM1} in the of Theorem~\ref{mainresult1}, therefore we omit it. 
Since~(ii) does not hold, there exists a characteristic root $\lambda\in\varSigma$ with $|\lambda|=1$. Evidently, $\lambda\in\varSigma^{cu}$, therefore the second relation in~\eqref{tphi} implies the existence of a nonzero vector $v\in\mathbb C^{s*}$ such that
\begin{equation}\label{veigen}
vB=\lambda v.
\end{equation}
Define $p\colon\mathbb Z^+_0\to\mathbb C^d$ by
\begin{equation}\label{pdef}
p(n)=\lambda^{n+1}(v\varPsi(0))^*,\qquad n\in\mathbb Z^+_0.	
\end{equation}
Since $\sup_{n\geq 0}|p(n)|=|(v\varPsi(0))^*|<\infty$, by Claim~\ref{fc}, the nonhomogeneous equation~\eqref{nvde} has at least one solution~$x$ on~$\mathbb Z^+_0$ such that~\eqref{bdd} holds. The corresponding coordinate function $z$ defined by~\eqref{zdef} satisfies the difference equation~\eqref{cueq}. From~\eqref{boudbil}, \eqref{Gamisom},  \eqref{zdef} and~\eqref{bdd}, we obtain that~$z$ and hence the function $u\colon\mathbb Z^+_0\to\mathbb C$ defined by
\begin{equation}
	u(n)=vz(n),\qquad n\in\mathbb Z^+_0,
\end{equation}
is bounded on~$\mathbb Z^+_0$. Multiplying Eq.~\eqref{cueq} from the left by~$v$ and using~\eqref{veigen}, we obtain that
\begin{equation}\label{ueq}
	u(n+1)=\lambda u(n)+v\langle\varPsi,\varGamma p(n)\rangle,\qquad n\in\mathbb Z^+_0.
\end{equation}
It follows from~\eqref{Gdef}, \eqref{pdef} and the definition of the bilinear form $\langle\cdot,\cdot\rangle$ that
\begin{equation*}
	\langle\varPsi,\varGamma p(n)\rangle=\varPsi(0)p(n)=\lambda^{n+1}\varPsi(0)(v\varPsi(0))^*.
\end{equation*}
This, together with~\eqref{ueq}, implies that
\begin{equation}\label{newueq}
	u(n+1)=\lambda u(n)+c\lambda^{n+1},\qquad n\in\mathbb Z^+_0,
\end{equation}
with
\begin{equation}\label{cdef}
c=(v\varPsi(0))(v\varPsi(0))^*=|(v\varPsi(0))^*|_2^2,
\end{equation}
where $|\cdot|_2$ denotes the $l_2$-norm on~$\mathbb C^d$. From Eq.~\eqref{newueq}, it follows by the variation of constants formula that
\begin{equation}\label{newestueq}
	u(n)=\lambda^{n}(u(0)+cn),\qquad n\in\mathbb Z^+_0.
\end{equation}
It follows from the duality~\eqref{dual} that $B=C$, where $B$ and~$C$ have the meaning from~\eqref{tphi} and~\eqref{tsharpsi}, respectively. Indeed, \eqref{dual} implies that
\begin{equation*}
	B=\langle \varPsi,\varPhi\rangle B
	=\langle\varPsi,\varPhi B\rangle
	=\langle\varPsi,T\varPhi\rangle
	=\langle T^\sharp\varPsi,\varPhi\rangle
	=\langle C\varPsi,\varPhi\rangle
	=C\langle\varPsi,\varPhi\rangle
	=C.
\end{equation*}
From~\eqref{tdef} and the relation $T\varPhi=\varPhi B$ (see~\eqref{tphi}), we find that
\begin{equation}\label{curep}
\varPhi(\theta)=\varPhi(0)B^\theta,\qquad\theta\in\mathbb Z^-_0.
\end{equation}
Similarly, from~\eqref{tsharpdef} and the relation $T^\sharp\varPsi=B\varPsi$ (see~\eqref{tsharpsi}), we have that
\begin{equation}\label{acurep}
\varPsi(\zeta)=B^{-\zeta}\varPsi(0),\qquad\zeta\in\mathbb Z^+_0.
\end{equation}
Next we will show that
\begin{equation}\label{nonzero}
	v\varPsi(0)\neq0.
\end{equation}
Suppose, for the sake of contradiction, that $v\varPsi(0)=0$. Then, by~\eqref{veigen} and~\eqref{acurep}, we have 
\begin{equation*}
v\varPsi(\zeta)=vB^{-\zeta}\varPsi(0)=\lambda^{-\zeta}v\varPsi(0)=0\qquad\text{for all $\zeta\in\mathbb Z^+_0$}.	
\end{equation*}
Thus, $v\varPsi$ is identically zero on~$\mathbb Z^+_0$. On the other hand, $v\neq0$ implies that $v\varPsi$ is a nontrivial linear combination of the basis elements $\psi_1,\dots,\psi_s$ of~$\mathcal N^\sharp$ and hence it cannot be identically zero on~$\mathbb Z^+_0$. This contradiction proves that~\eqref{nonzero} holds.

From~\eqref{cdef} and~\eqref{nonzero}, we find that $c>0$. From this and~\eqref{newestueq}, taking into account that $|\lambda|=1$, we obtain that
\begin{equation*}
	|u(n)|=|u(0)+cn|\rightarrow\infty,\qquad n\to\infty,
\end{equation*}
which contradicts the boundedness of~$u$.

$(ii) \Rightarrow (i)$. Suppose that
the characteristic equation~\eqref{chareq} has no root with $|\lambda|=1$. Then, the exponential estimates of the solution operator on the stable and unstable subspaces of~$\mathcal B_\gamma$ (see, e.g., \cite[Theorem~1]{N}) imply that Eq.~\eqref{vde} admits an exponential dichotomy on~$\mathbb Z^+_0$ as defined in~\cite{DP} with projections $P_n=\varPi^s$ for $n\in\mathbb Z^+_0$. By the application of \cite[Theorem~1]{DP} with $f_n=0$, $c=0$ and $\mu=1$, we conclude that Eq.~\eqref{vde} is shadowable on~$\mathbb Z^+_0$.
\end{proof}

\begin{remark}\label{r5}
It is known that certain solutions of Eq.~\eqref{vde} can be continued backward in the sense that they satisfy Eq.~\eqref{vde} for all~$n\in\mathbb Z$. Such solutions are sometimes called \emph{global solutions} or \emph{entire solutions}.
In a recent paper~\cite{BV}, Barreira and Valls have considered the Ulam--Hyers stability, a special case of shadowing, of the global solutions for difference equations with finite delays. In \cite[Theorem~5]{BV}, they have proved the analogue of a recent shadowing theorem~\cite[Theorem~1]{DP} for global solutions. Moreover, in the special case of linear autonomous equations in finite dimensional spaces they have shown that the Ulam--Hyers stability of the global solutions is equivalent to the existence of an exponential dichotomy whenever the coefficients are scalar (see \cite[Theorem~8]{BV}) or the Jordan blocks associated with the central directions are diagonal (see \cite[Theorem~9]{BV}). 
Using a similar argument as in the proof of Theorem~\ref{mainresult2}, it can be shown that the assumption about the Jordan blocks in \cite[Theorem~9]{BV} is superfluous.
In this sense, Theorem~\ref{mainresult2} may be viewed as an improvement of \cite[Theorem~9]{BV}. 
\end{remark}

\section*{Acknowledgements}
L. Backes was partially supported by a CNPq-Brazil PQ fellowship under Grant No. 307633/2021-7.
D.~Dragi\v cevi\' c was supported in part by University of Rijeka under the project uniri-iskusni-prirod-23-98
3046.
M.~Pituk was supported by the Hungarian National Research, Development and Innovation Office grant no.~K139346.

\end{document}